\documentclass[11 pt]{amsart}

\usepackage{amsfonts, amsmath, amssymb, amsthm, amsopn, latexsym, hyperref, graphicx, mathrsfs, verbatim, enumerate, url, stmaryrd} 
\usepackage{subfig} 
\usepackage[all]{hypcap} 
\usepackage{pst-all, multido} 
\usepackage[margin=1in]{geometry} 
\usepackage{Style} 
\usepackage{pspicture}

\newtheorem{letthm}{Theorem}

\newtheorem{thm}{Theorem}[section]
\newtheorem{lem}[thm]{Lemma}
\newtheorem{cor}[thm]{Corollary}
\newtheorem{prop}[thm]{Proposition}

\theoremstyle{definition}
\newtheorem{rem}[thm]{Remark}
\newtheorem{defi}[thm]{Definition}
\newtheorem{ex}[thm]{Example}

\title[Growth of attraction rates for iterates of a superattracting germ]{Growth of attraction rates for iterates of a superattracting germ in dimension two}

\date{\today}

\author[W. Gignac]{William Gignac}
\address{Department of Mathematics, University of Michigan, 530 Church St., Ann Arbor, MI 48109, USA}
\email{wgignac@umich.edu}

\author[M. Ruggiero]{Matteo Ruggiero}
\address{Fondation Math\'ematique Jacques Hadamard --
Centre de Math\'ematiques Laurent Schwartz, 
\'Ecole Polytechnique, 
91128 Palaiseau Cedex, France.}
\email{ruggiero@math.polytechnique.fr}

\begin{document}

\maketitle

\begin{abstract} We study the sequence of attraction rates of iterates of a dominant superattracting holomorphic fixed point germ $f\colon (\C^2, 0)\to (\C^2, 0)$. By using valuative techniques similar to those developed by Favre-Jonsson, we show that this sequence eventually satisfies an integral linear recursion relation, which, up to replacing $f$ by an iterate, can be taken to have order at most two. In addition, when the germ $f$ is finite, we show the existence of a bimeromorphic model of $(\C^2,0)$ where $f$ satisfies a weak local algebraic stability condition. \end{abstract}

\section*{Introduction}

Let $f\colon (\C^2,0)\to (\C^2,0)$ be a holomorphic fixed point germ that is dominant (i.e., whose image is not contained in any curve). In local coordinates $(x,y)$ at $0$, the map $f$ can be expanded in a power series of the form $f(x,y) = f_c(x,y) + f_{c+1}(x,y) + \cdots$, where the $f_i$ are homogeneous polynomials of degree $i$ and $f_c\neq 0$. The integer $c = c(f)\geq 1$ is independent of the choice of coordinates $(x,y)$. In this article we will study the sequence $c(f^{n})$ in the case where $f$ is superattracting, that is, when $c(f^n)\to \infty$. Superattracting germs are an active area of research in complex dynamics in several variables, see for instance \cite{MR1275463}, \cite{favre:rigidgerms}, \cite{favre-jonsson:eigenval}, \cite{BEK}, \cite{casasalvero-roe:iteratedinverseimages}, \cite{MZMS}, \cite{gignac:local}, \cite{Rug2}, \cite{ruggiero:rigidification},  and the notes \cite{jonsson:berkovich}.

From the viewpoint of complex dynamics, the sequence $c(f^n)$ is interesting because it gives a measure of the rate at which nearby points are attracted to $0$ under iteration. More precisely, one has that $c(f^n) = \max\{c>0 : \|f(p)\| = O(\|p\|^c)\mbox{ as }p\to 0\}$. For this reason, we call $c(f^n)$ the \emph{attraction rate} of $f^n$. The limit $c_\infty := \lim_{n\to \infty} c(f^n)^{1/n}$, which we call the \emph{asymptotic attraction rate} of $f$, measures the growth of the sequence $c(f^n)$ of attraction rates. Favre-Jonsson show in \cite{favre-jonsson:eigenval} that the quantity $c_\infty$ is a quadratic integer (see also \cite{jonsson:berkovich}), a fact which suggests some regularity in the sequence $c(f^n)$. Our main result confirms this regularity.

\begin{letthm}\label{thm:recursionrelation} Let $f:(\C^2, 0) \rightarrow (\C^2, 0)$ be a dominant, holomorphic, superattracting  fixed point germ. Then the sequence $c(f^n)$ eventually satisfies an integral linear recursion relation. Moreover, up to replacing $f$ by an iterate, the recursion relation can be taken to have order at most $2$.
\end{letthm}

When $f$ is a \emph{finite} germ, that is, when $f$ is finite-to-one near $0$, \hyperref[thm:recursionrelation]{Theorem~\ref*{thm:recursionrelation}} has a closely related geometric counterpart (\hyperref[thm:weakstability]{Theorem~\ref*{thm:weakstability}}), owing to the fact that the attraction rates $c(f^n)$ can be computed as intersection numbers in certain bimeromorphic models of $(\C^2,0)$.
Before being able to state \hyperref[thm:weakstability]{Theorem~\ref*{thm:weakstability}}, we need a small amount of terminology.
In this article, a \emph{modification} of  $(\C^2,0)$ is a proper, bimeromorphic map $\pi\colon X\to (\C^2,0)$ that is a biholomorphism over $\C^2\smallsetminus\{0\}$; a modification $\pi'\colon X'\to (\C^2,0)$ \emph{dominates} a modification $\pi\colon X\to (\C^2,0)$ if there is a holomorphic map $\mu\colon X'\to X$ such that $\pi \circ \mu  = \pi'$.
Given a modification $\pi\colon X\to (\C^2,0)$, we let $\Div_\Q(\pi)$ denote the $\Q$-vector space generated by the irreducible components of $\pi^{-1}(0)$.
The meromorphic maps $f^n\colon X\dashrightarrow X$ induce $\Z$-linear pullback maps $(f^n)^*\colon \Div_\Q(\pi)\to \Div_\Q(\pi)$.
In general these pullbacks won't be functorial in the sense that $(f^n)^* = (f^*)^n$, but \hyperref[thm:weakstability]{Theorem~\ref*{thm:weakstability}} shows that one can always choose the modification $\pi$ in such a way that the pullbacks $(f^n)^*$ are ``eventually functorial".

\begin{letthm}\label{thm:weakstability}
Let $f\colon (\C^2,0)\to (\C^2,0)$ be a finite, holomorphic, superattracting fixed point germ.
Then there is a modification $\pi\colon X\to (\C^2,0)$ dominating the blowup of the origin in $\C^2$, and an integer $N\geq 0$ such that
\[(f^n)^* = (f^*)^{n-N}(f^N)^*\]
on $\Div_\Q(\pi)$ for any $n>N$.
Moreover, $X$ can be taken to have at worst Hirzebruch-Jung quotient singularities, and, if $f$ is replaced by $f^2$, one can take $X$ to be smooth.
\end{letthm}

In any modification $\pi\colon X\to (\C^2,0)$ dominating the blowup of the origin, there is a divisor $D\in \Div_\Q(\pi)$ such that $c(f^n) = -D\cdot (f^n)^*D$ for all $n\geq1$.
The fact that the attraction rates $c(f^n)$ eventually satisfy an integral linear recursion formula then follows immediately from \hyperref[thm:weakstability]{Theorem~\ref*{thm:weakstability}}, though the fact that, up to replacing $f$ by an iterate, this relation can be taken to have order at most two is not obvious.

In general, we will show by example (see \hyperref[ex:not_stable]{Example~\ref*{ex:not_stable}}) that it is not always possible to find a modification $\pi\colon X\to (\C^2,0)$ for which functoriality $(f^n)^* = (f^*)^n$ holds for all $n\geq 1$.
Finding such an $X$ is a local analogue of finding an \emph{algebraically stable} model for a  global complex dynamical system, a very active area of research, see, for instance, \cite{MR1867314}, \cite{MR2021001},  \cite{MR2140266}, \cite{MR2129771}, \cite{MR2545825}, \cite{favre-jonsson:dynamicalcompactifications}, \cite{MR2861092}, \cite{MR2917145}, and \cite{Lin}.
In this light, \hyperref[thm:weakstability]{Theorem~\ref*{thm:weakstability}} should be viewed as a guarantee that models always exist on which $f$ satisfies a weak local algebraic stability condition.

The techniques used in this article are not complex analytic.
Indeed, Theorems \ref{thm:recursionrelation} and \ref{thm:weakstability} remain valid when one replaces $\C$ by any field $k$ of characteristic $0$ (without loss of generality algebraically closed, since working over $\ol{k}$ does not change the statements) and $f$ by any formal fixed point germ.
In purely algebraic terms, one can rephrase \hyperref[thm:recursionrelation]{Theorem~\ref*{thm:recursionrelation}} as follows.
Let $(R,\mf{m})$ be a complete regular local ring of dimension $2$ and residue characteristic $0$, and suppose that $f\colon \mathrm{Spec}(R)\to \mathrm{Spec}(R)$ is a dominant morphism which is ``superattracting" in the sense that there is an integer $r \geq 1$ for which $(f^r)^*\mf{m}\subseteq\mf{m}^2$.
Then the sequence $c(f^n):= \max\{ i : (f^n)^*\mf{m}\subseteq\mf{m}^i\}$ eventually satisfies an integral linear recursion relation.
It should be noted that, while Theorems \ref{thm:recursionrelation} and \ref{thm:weakstability} may be true over fields of characteristic $p>0$ as well, we will use the characteristic $0$ hypothesis in an essential way, see \hyperref[rem:JacCharp]{Remarks~\ref*{rem:JacCharp}} and \ref{rem:charp}.

To prove Theorems \ref{thm:recursionrelation} and \ref{thm:weakstability}, we will use the valuative techniques developed by Favre-Jonsson in \cite{favre-jonsson:eigenval}, \cite{favre-jonsson:valtree}, and \cite{favre-jonsson:dynamicalcompactifications}.
In fact, analogues of our theorems are proved in \cite{favre-jonsson:dynamicalcompactifications} for the sequence $\deg(F^{n})$, where $F\colon \C^2\to \C^2$ is a polynomial map; our theorems should therefore be viewed as local versions of those global results.
We will deduce Theorems \ref{thm:recursionrelation} and \ref{thm:weakstability} from an analysis of the dynamics induced by $f$ on a certain space $\mc{V}$ of valuations.
It is shown in \cite{favre-jonsson:eigenval} that there are fixed points for these dynamics which are in some sense attracting.
Much of the work done in this article will be in showing that the basins of attraction of these fixed points are large.
To prove this, the main tool we use is a theorem regarding the equicontinuity of the dynamics on $\mc{V}$, see \hyperref[thm:equicontinuity]{Theorem~\ref*{thm:equicontinuity}}.
This approach differs from the one taken in \cite{favre-jonsson:dynamicalcompactifications}, where the authors use the Hilbert space methods of \cite{boucksom-favre-jonsson:degreegrowthmeromorphicsurfmaps}, which unfortunately do not carry over well to our local setting.

Finally, we should note that the sequence $c(f^n)$ has recently appeared in the work of Casas-Alvero and Ro\'{e} \cite{casasalvero-roe:iteratedinverseimages}, where \hyperref[thm:recursionrelation]{Theorem~\ref*{thm:recursionrelation}} is proved for a specific class of germs $f$ by relating the sequence $c(f^n)$ to the behavior of certain local intersection numbers under the dynamics of $f$.
This highlights a link between the growth of attraction rates $c(f^n)$ and problems relating to the growth of local intersection numbers (see \cite[\S5]{arnold:probltheoriesystemedynamiques}, \cite{arnold:boundsmilnornumbersintersections}, and \cite{gignac:local}). 

The basic outline of this article is as follows.
First, in \S1 and \S2, we will review some background on the valuation space $\mc{V}$, called the \emph{valuative tree}, on which we will be working.
In \S1, the focus is on describing the structure of $\mc{V}$, while in \S2 we discuss the dynamics induced by $f$ on $\mc{V}$.
The heart of this article is in \S3 and \S4, where we show that there is a collection of fixed points for the dynamics of $f$ on $\mc{V}$ which attract most points of $\mc{V}$ under iteration.
In \S5 and \S6 we will use the results of \S3 and \S4 to prove Theorems \ref{thm:recursionrelation} and \ref{thm:weakstability}.
Finally, \S7 is devoted to worked examples.
 
\subsection*{Acknowledgements}
The authors would like to thank Charles Favre and Mattias Jonsson for their invaluable guidance, insight, and encouragement during the course of this project.
The first author was supported by the grants DMS-0602191, DMS-0901073, and DMS-1001740; moreover, a portion of this work was done while the first author was a visiting researcher at the Institute for Computational and Experimental Research Mathematics.
The second author was supported by a bourse of the Fondation Math\'ematique Jacques Hadamard.

\section{The valuative tree}

In this section we give an overview of the structure of the \emph{valuative tree} $\mc{V}$, a space of valuations on which we will be studying dynamics.
There are two aspects of this structure we will emphasize.
First, $\mc{V}$ is a combinatorial object; specifically, it is an ordered tree.
Second, $\mc{V}$  is a geometric object in the sense that it encodes much of the local geometry at $0\in \C^2$.
The most detailed references for this subject are the monograph \cite{favre-jonsson:valtree} and the notes \cite{jonsson:berkovich}.
The former stresses the combinatorial aspects of $\mc{V}$, while the latter prefers the geometric approach, a perspective that has been useful in applications, see \cite{favre-jonsson:dynamicalcompactifications}, \cite{favre:holoselfmapssingratsurf} for dynamical applications, \cite{favre-jonsson:valanalplanarplurisubharfunct}, \cite{boucksom-favre-jonsson:valuationsplurisubharsing} for applications to singularities of plurisubharmonic functions, and \cite{favre-jonsson:valmultideals}, \cite{jonsson-mustata:asymptotic}, \cite{jonsson-mustata:openness}, \cite{hu:multiplier}, \cite{hu:logcanonical} for applications to multiplier ideals in analytic and algebraic geometry. 

\subsection{Definition and tree structure}

The valuative tree $\mc{V}$ is a subspace of a certain space $\hat{\mc{V}}^*$ of semivaluations on the ring of formal power series $\C\llbracket x,y\rrbracket$.
Throughout this article, we will denote by $\mf{m}$ the unique maximal ideal of $\C\llbracket x,y\rrbracket$. 

\begin{defi}
A \emph{centered semivaluation} on $\C\llbracket x,y\rrbracket$ is a map $\nu\colon\C\llbracket x,y\rrbracket\to [0,+\infty]$ such that
\begin{enumerate}
\item[1.] $\nu(0) = +\infty$ and $\nu|_{\C^*} \equiv 0$,
\item[2.] $\nu(\phi\psi) = \nu(\phi) + \nu(\psi)$ for all $\phi,\psi\in \C\llbracket x,y\rrbracket$,
\item[3.] $\nu(\phi + \psi)\geq \min\{\nu(\phi), \nu(\psi)\}$ for all $\phi,\psi\in \C\llbracket x,y\rrbracket$, and 
\item[4.] $\nu(\phi)>0$ if and only if $\phi\in \mf{m}$.
\end{enumerate}
Given such a $\nu$ and any ideal $\mf{a}\subseteq \C\llbracket x,y\rrbracket$, we let $\nu(\mf{a}) := \min\{\nu(\phi) : \phi\in \mf{a}\}$.
It is easy to see that the minimum exists; indeed, one need only take the minimum over a (finite) set of generators of $\mf{a}$.

Define $\hat{\mc{V}}^*$ to be the set of all centered semivaluations $\nu$ such that $\nu(\mf{m})<\infty$, and $\mc{V}$ to be the subset consisting of those $\nu$ with $\nu(\mf{m}) = 1$.
The set $\mc{V}$ is called the \emph{valuative tree}.
Note that positive scalar multiples $\lambda \nu$ of centered semivaluations $\nu$ are again centered semivaluations.
From this it follows that $\hat{\mc{V}}^*$ is a cylinder over $\mc{V}$, i.e., $\hat{\mc{V}}^*\cong \mc{V}\times (0, +\infty)$.
\end{defi}

\begin{ex}
Perhaps the most important example of a semivaluation in $\mc{V}$ is the $\mf{m}$-adic valuation $\ord_0$, which is given by $\ord_0(\phi) := \max\{k\in \N : \phi\in \mf{m}^k\}$ for all $\phi\in \C\llbracket x,y\rrbracket$.
It has the property that $\ord_0(\phi)\leq \nu(\phi)$ for all $\nu\in \mc{V}$ and all $\phi\in \C\llbracket x,y\rrbracket$.
\end{ex}

We equip $\hat{\mc{V}}^*$ with the partial order $\leq$ defined by setting $\nu\leq\mu$ if $\nu(\phi)\leq \mu(\phi)$ for all $\phi\in \C\llbracket x,y\rrbracket$.
Since $\mc{V}\subset\hat{\mc{V}}^*$, this partial order restricts to $\mc{V}$.
The utility of the valuative tree $\mc{V}$ lies partly in its rich poset structure under $\leq$.
It is a nontrivial fact (see \cite[\S3.2]{favre-jonsson:valtree} or \cite[\S7.6]{jonsson:berkovich}) that $(\mc{V},\leq)$ is a \emph{complete rooted tree} in the sense that
\begin{enumerate}
\item[1.] $\mc{V}$ has a unique minimal element, or \emph{root}, namely $\ord_0$,
\item[2.] for any $\nu \neq \ord_0$, the set $\{\mu \in \mc{V} : \mu \leq \nu\}$ is order isomorphic to a closed interval in $\R$,
\item[3.] any two points $\nu,\mu\in \mc{V}$ admit an infimum $\nu\wedge \mu\in \mc{V}$, and
\item[4.] any totally ordered subset of $\mc{V}$ has a least upper bound in $\mc{V}$.
\end{enumerate}

In addition to this poset structure, $\hat{\mc{V}}^*$ has a natural topology, namely the weakest topology for which each of the evaluation maps $\mathrm{ev}_\phi\colon \hat{\mc{V}}^*\to [0,+\infty]$ given by $\nu\mapsto \nu(\phi)$ are continuous.
This is called the \emph{weak topology}.
The valuative tree $\mc{V}$ is compact Hausdorff in the weak topology, and $\hat{\mc{V}}^*\cong \mc{V}\times (0, +\infty)$ is locally compact Hausdorff.
The weak topology is not metrizable.

\subsection{Classification of points of $\hat{\mc{V}}^*$} 

There is a classification of all semivaluations $\nu\in \hat{\mc{V}}^*$ into four different types, called \emph{divisorial}, \emph{irrational}, \emph{curve}, and \emph{infinitely singular} valuations.
The details of this classification will not be needed in this work, so we restrict ourselves to a rough outline and refer to \cite[\S2.2]{favre-jonsson:valtree} and \cite[\S7.7]{jonsson:berkovich} for details.
Throughout this article when we refer to a blowup $\pi\colon X\to \C^2$ of the plane over $0$, we mean a composition of point blowups over $0$.
By an \emph{exceptional prime} of $\pi$, we mean a component of $\pi^{-1}(0)$.

\subsubsection*{Divisorial valuations}
Let $\pi\colon X\to \C^2$ be a blowup of the plane over $0$, and let $E$ be an exceptional prime of $\pi$.
From this information one can construct a valuation $\ord_E$ on $\C\llbracket x,y\rrbracket$ as follows.
Let $\mc{O}_{X,E}$ be the local ring of $X$ at $E$, and let $\mf{p}_E$ be the maximal ideal in this ring.
For any polynomial $\phi\in \C[x,y],$ define $\ord_E(\phi) := \max\{k\in \N : \phi\circ \pi \in \mf{p}_E^k\}$.
This $\ord_E$ defines a valuation on $\C[x,y]$ with the property that $\ord_E(\phi)\geq 0$ for all $\phi\in \C[x,y]$ and $\ord_E(\phi)>0\iff \phi\in \mf{m}$.
This property guarantees that $\ord_E$ extends uniquely to a valuation $\ord_E$ on the $\mf{m}$-adic completion $\C\llbracket x,y\rrbracket$ of $\C[x,y]$.

Any valuation $\nu\in \hat{\mc{V}}^*$ of the form $\lambda \ord_E$ for some $\lambda>0$ is called \emph{divisorial}.
The valuation $\ord_E$ will not in general lie in $\mc{V}$, and must be normalized to give a valuation in $\mc{V}$.
We will always denote by $b_E$ the constant such that $b_E^{-1}\ord_E\in \mc{V}$.
The $\mf{m}$-adic valuation $\ord_0$ is an example of a divisorial valuation in $\mc{V}$.
Indeed, $\ord_0 = \ord_{E}$ where $E$ is the exceptional prime obtained from blowing up the origin in $\C^2$ a single time.

\subsubsection*{Irrational valuations}
Let $\pi\colon X\to \C^2$ be a blowup of the plane over $0$, and suppose that $E$ and $F$ are two exceptional primes of $\pi$ which intersect at a point $p$.
Let $(z,w)$ be local coordinates of $X$ at $p$ such that $E = \{z = 0\}$ and $F = \{w = 0\}$, and let $r,s>0$ be real numbers.
From this information one can construct a valuation $\nu$ on $\C\llbracket x,y\rrbracket$ as follows.
For any $\phi\in \C\llbracket x,y\rrbracket$, write $\phi\circ \pi$ as a formal power series in the coordinates $(z,w)$, say $\phi\circ \pi = \sum \lambda_{ij}z^iw^j$.
One then defines $\nu(\phi) := \min\{ri + sj : \lambda_{ij}\neq 0\}$. 

Such a valuation $\nu$ always lies in $\hat{\mc{V}}^*$.
If $r$ and $s$ are rationally dependent, it turns out that $\nu$ will be a divisorial valuation.
However, if $r$ and $s$ are rationally independent, then $\nu$ is not divisorial, and instead we call it an \emph{irrational valuation}.

\subsubsection*{Curve valuations}
Suppose that $\phi\in \mf{m}$ is an irreducible element of $\C\llbracket x,y\rrbracket$.
It defines an irreducible formal curve germ $C := \{\phi = 0\}$ in $\C^2$ at $0$.
One then obtains a semivaluation in $\nu\in \hat{\mc{V}}^*$ in the following manner: for $\psi\in \C\llbracket x,y\rrbracket$, we define $\nu(\psi)$ to be the order of vanishing of the restriction $\psi|_C$ at the origin.
The semivaluation $\nu$ takes the value $+\infty$ precisely on the ideal $(\phi)$.
In general $\nu\notin \mc{V}$; we will denote by $\nu_\phi$ the normalized semivaluation $\nu_\phi = \nu(\mf{m})^{-1}\nu\in \mc{V}$.
Abusing terminology slightly, any semivaluation of the form $\lambda\nu_\phi$ is called a \emph{curve valuation}, even though strictly speaking it is not a valuation.

\subsubsection*{Infinitely singular valuations}
Any $\nu\in \hat{\mc{V}}^*$ that does not fit into one of the previous three categories is called an \emph{infinitely singular valuation}.
These have an interpretation in terms of generalized Puiseux series (see \cite[\S1.5.7]{favre-jonsson:valtree}), but we omit the details here.
We note, however, that infinitely singular valuations are always valuations instead of just semivaluations.

\subsubsection*{Quasimonomial valuations and ends}
An element $\nu\in \hat{\mc{V}}^*$ is called \emph{quasimonomial} if it is divisorial or irrational, and is called an \emph{end} if it is a curve valuation or an infinitely singular valuation.
The latter terminology alludes to the fact that curve valuations and infinitely singular valuations are exactly the ends of the tree $\mc{V}$, i.e., the maximal elements in the partial order $\leq$.

\subsection{The valuative tree and geometry}
We have just seen that quasimonomial valuations arise from geometric data, namely, from the geometry of blowups over $0$.
In this way, $\hat{\mc{V}}^*$ encodes information about the local geometry of $\C^2$ at $0$.
We now briefly sketch another way that quasimonomial valuations interact with the local geometry of $\C^2$ at $0$.

Let $\nu\in \hat{\mc{V}}^*$ be a quasimonomial valuation and let $\pi\colon X\to \C^2$ be a blowup of the plane over $0$.
We denote by $\Div(\pi)$ the free abelian group on the exceptional primes of $\pi$.
One can evaluate $\nu$ on divisors $D\in \Div(\pi)$ in the following way.
Let $p\in X$ be the \emph{center} of $\nu$ in $X$.
Then $\nu$ defines a valuation on the local ring $\mc{O}_{X,p}$.
If $g$ is a local defining equation of $D$ at $p$, we set $\nu(D) := \nu(g)$.
Observe that $\nu\colon \Div(\pi)\to \R$ is $\Z$-linear.
Since the intersection pairing on $\Div(\pi)$ is non-degenerate, it follows that there is an $\R$-divisor $Z_{\nu,\pi}\in \R\otimes_\Z\Div(\pi)$ such that $\nu(D) = (Z_{\nu,\pi}\cdot D)$.
In this way $\nu$ defines a divisor $Z_{\nu,\pi}$ on every blowup $\pi\colon X\to \C^2$ over $0$.
This construction is carried out in detail in \cite[\S1.2]{favre:holoselfmapssingratsurf} and (in a slightly different setting) in the appendix of \cite{favre-jonsson:dynamicalcompactifications}.

In this article, we will only use this construction in the case where $\nu$ is divisorial, so we restrict to that case now.
If $\pi\colon X\to \C^2$ is a blowup over $0$ such that $\nu = \lambda \ord_E$ for some exceptional prime $E$ of $X$, then $Z_{\nu, \pi} = \lambda \check{E}$, where $\check{E}\in \Div(\pi)$ is the unique divisor such that $(\check{E}\cdot E) = 1$ and $(\check{E}\cdot F) = 0$ for all exceptional primes $F\neq E$ of $\pi$.
If $\pi'\colon X'\to \C^2$ is a blowup over $0$ that dominates $\pi$, that is, if there is a holomorphic map $\mu\colon X'\to X$ such that $\pi' = \pi\circ \mu$, then $Z_{\nu,\pi} = \mu_*Z_{\nu,\pi'}$ and $Z_{\nu,\pi'} = \mu^*Z_{\nu,\pi}$.
Finally, we note that for any $\pi$, one has $\nu(\mf{m}) = -(Z_{\nu,\pi}\cdot Z_{\ord_0,\pi})$.


\subsection{Tangent vectors}
Let $\nu,\mu\in \mc{V}$ be any two semivaluations.
The \emph{segment between $\nu$ and $\mu$} is the set $[\nu,\mu]\subset \mc{V}$ defined as follows.
If $\nu\leq \mu$, then $[\nu,\mu]:=\{\eta\in \mc{V} : \nu\leq \eta\leq \mu\}$.
For general $\nu,\mu$, one sets $[\nu,\mu]:= [\nu\wedge\mu, \nu]\cup[\nu\wedge\mu, \mu]$.
The segment $[\nu,\mu]$ is the minimal path in $\mc{V}$ connecting $\nu$ and $\mu$.
In the weak topology, it is homeomorphic to $[0,1]$ whenever $\nu\neq \mu$.

\begin{defi}
A \emph{subtree} of $\mc{V}$ is a subset $T\subseteq \mc{V}$ such that if $\nu,\mu\in T$, then $[\nu,\mu]\subseteq T$.
\end{defi}

\begin{defi}
Fix a semivaluation $\nu\in \mc{V}$.
Define an equivalence relation on $\mc{V}\setminus\{\nu\}$ by declaring $\mu_1$ and $\mu_2$ equivalent if $[\nu, \mu_1]\cap [\nu,\mu_2]\neq \{\nu\}$.
An equivalence class under this relation is called a \emph{tangent vector} at $\nu$.
The set of all such equivalence classes is the \emph{tangent space} of $\mc{V}$ at $\nu$.
\end{defi}

The tangent space at $\nu$ should be thought of as the collection of ``branches" in $\mc{V}$ leaving $\nu$.
Two semivaluations $\mu_1,\mu_2\in \mc{V}$ are representatives of the same tangent vector at $\nu$ if and only if they belong to the same branch leaving $\nu$.
It is typical to denote a tangent vector using the vector notation $\vec{v}$, and the set of all $\mu$ in the direction of $\vec{v}$ by $U(\vec{v})$.
The sets $U(\vec{v})$ are weakly open. 

If $\nu\in \mc{V}$ is a curve valuation or an infinitely singular valuation, then the tangent space at $\nu$ consists of only one tangent vector, because such $\nu$ are ends of the valuative tree.
The tangent space of an irrational valuation $\nu$ consists of exactly two tangent vectors.
Informally, irrational valuations are not branching points of $\mc{V}$.
The only branching points of $\mc{V}$ are the divisorial valuations.
The tangent space at a divisorial valuation is quite large: if $\nu = b_E^{-1}\ord_E$ is divisorial, then the tangent space at $\nu$ is in one-to-one correspondence with the points of $E$ (see \cite[Theorem B.1]{favre-jonsson:valtree}).

\subsection{Parameterizations}
While the weak topology is the most natural topology on $\mc{V}$, it will prove useful to consider other topologies, induced by \emph{parameterizations} of $\mc{V}$.

\begin{defi}
A \emph{parameterization} on $\mc{V}$ is a monotone function $a\colon \mc{V}\to [-\infty,\infty]$ such that the restriction of $a$ to any segment $[\nu,\mu]$ with $\nu\leq \mu$ is either an order preserving or order reversing isomorphism onto a closed subinterval of $[-\infty,\infty]$.
\end{defi}

Given a parameterization $a$ on $\mc{V}$, we will denote by $\mc{V}_a$ the collection of all semivaluations $\nu\in \mc{V}$ such that $|a(\nu)|<\infty$.
The parameterization induces a metric $d_a$ on $\mc{V}_a$, defined by
\[
d_a(\nu,\mu) := |a(\nu) - a(\nu\wedge\mu)| + |a(\mu) - a(\nu\wedge\mu)|.
\]
The topology given by such a metric is strictly stronger than the weak topology.
Two very useful parameterizations of $\mc{V}$ are the \emph{thinness} and \emph{skewness} parameterizations, defined below.

\begin{defi}\label{def:skewthin}
The \emph{skewness} parameterization $\alpha\colon \mc{V}\to [1,+\infty]$ is defined by
\[\alpha(\nu)  := \sup_{\phi\in \mf{m}} \nu(\phi)/\ord_0(\phi).\]
The \emph{thinness} parameterization $A\colon \mc{V}\to [2,+\infty]$ is defined by 
\[A(\nu) := 2 + \int_{\ord_0}^\nu m(\mu)\,d\alpha(\mu),\]
where $m(\mu) := \min\{\ord_0(\phi): \phi\in \mf{m}$ is irreducible and $\mu\leq \nu_\phi\}$.
\end{defi}

A quasimonomial valuation $\nu$ always has finite thinness and skewness.
Moreover, the thinness and skewness are rational if $\nu$ is divisorial, and irrational if $\nu$ is irrational.
A curve semivaluation always has infinite thinness and skewness.
In general nothing can be said about the finiteness of the thinness or skewness of an infinitely singular valuation.
It follows from the definition of thinness that $\alpha(\nu)\leq A(\nu)$ for all $\nu\in \mc{V}$.
One of the most important properties of the skewness, proved in \cite[Proposition 3.25]{favre-jonsson:valtree}, is the following. 

\begin{prop}\label{prop:skewness_eval}
Let $\nu\in \mc{V}$, and let $\phi\in \C\llbracket x,y\rrbracket$ be irreducible.
Then $\nu(\phi) = \alpha(\nu\wedge \nu_\phi)\ord_0(\phi)$, where $\nu_\phi$ is the curve semivaluation corresponding to $\phi$.
\end{prop}

It should be noted that these are the definitions of thinness and skewness given in \cite{favre-jonsson:valtree}.
It is possible, and in some settings preferable, to give equivalent definitions of thinness and skewness which are of a more geometric flavor.
This is the approach taken in the notes \cite{jonsson:berkovich}, where the term \emph{thinness} is replaced by \emph{log discrepancy} and the definition of skewness differs from ours by a sign, as well as in the paper \cite{favre:holoselfmapssingratsurf}.
In this paper, we will only need the geometric interpretation of skewness for divisorial valuations.
Recall from \S1.3 that if $\pi\colon X\to \C^2$ is a blowup over $0$ and if $\nu\in \mc{V}$ is divisorial, then $\nu$ induces a divisor $Z_{\nu,\pi}$ on $X$.
If $\nu = b_E^{-1}\ord_E$ for some exceptional prime $E$ of $\pi$, then the skewness of $\alpha$ is given by the self-intersection $\alpha(\nu) = -(Z_{\nu,\pi})^2$.
More generally, if $\mu = b_F^{-1}\ord_F$ and $\nu = b_E^{-1}\ord_E$ where $E$ and $F$ are exceptional primes of $\pi$, then $\alpha(\mu\wedge\nu) = -(Z_{\mu,\pi}\cdot Z_{\nu,\pi})$, see \cite[Lemma A.2]{boucksom-favre-jonsson:degreegrowthmeromorphicsurfmaps}.


\section{Dynamics on the valuative tree}

Suppose that $f\colon (\C^2,0)\to (\C^2,0)$ is a holomorphic fixed point germ, with $f = (f_1,f_2)$.
Such a germ is \emph{dominant} if its Jacobian determinant $J_f\in \C\llbracket x,y\rrbracket$ is nonzero.
It is said to be \emph{superattracting} if the derivative $f'(0)$ of $f$ at $0$ is a nilpotent linear map.
In this and later sections we will only concern ourselves with dominant superattracting germs $f$.
Fix such an $f$ now.

Given a $\nu\in \hat{\mc{V}}^*$, one can push forward $\nu$ by $f$ to obtain a centered semivaluation $f_*\nu$, given by $(f_*\nu)(\phi):= \nu(\phi\circ f)$.
In general the semivaluation $f_*\nu$ will not belong to $\hat{\mc{V}}^*$, since one can have that $(f_*\nu)(\mf{m}) = +\infty$.
The value $(f_*\nu)(\mf{m})$ is of central importance in this article, enough so that it warrants a name.

\begin{defi} Let $\nu\in \mc{V}$.
The \emph{attraction rate of $f$ along $\nu$} is the constant $c(f,\nu) := (f_*\nu)(\mf{m})$, which lies in $[1,+\infty]$.
\end{defi}

\begin{ex}
As an example, consider the particular case when $\nu = \ord_0$.
In this case, $c(f,\ord_0) = (f_*\ord_0)(\mf{m}) = \min\{\ord_0(f_1), \ord_0(f_2)\}$ is exactly the quantity $c(f)$ defined in the introduction.
Thus \hyperref[thm:recursionrelation]{Theorem~\ref*{thm:recursionrelation}} is a statement about the sequence of attraction rates $c(f^n,\ord_0)$.
\end{ex}

One easily checks that a valuation $\nu\in \mc{V}$ has $c(f,\nu) = +\infty$ if and only if $\nu$ is a curve valuation $\nu = \lambda \nu_\phi$ such that the formal curve $C = \{\phi = 0\}$ is contracted to $0$ under $f$.
We call such a $\nu$ a \emph{contracted curve valuation}; there are at most finitely many contracted curve valuations for $f$.
See \cite[\S2]{favre-jonsson:eigenval} for an in depth discussion.
The map $c(f,-)\colon \mc{V}\to [1,+\infty]$ is continuous, and has the following useful properties.

\begin{prop}[{\cite[Proposition 3.4]{favre-jonsson:eigenval}}]\label{prop:loc_const}
For any  $\nu\in \mc{V}$, the map $c(f,-)\colon [\ord_0, \nu]\to [1,+\infty]$ is increasing, concave, and piecewise affine with respect to the skewness parameterization.
Let $T$ be the set of $\nu\in \mc{V}$ such that $c(f,-)$ is not locally constant near $\nu$.
Then $T$ is a connected closed subtree of $\mc{V}$ which is finite in the sense that it has finitely many ends.
\end{prop}

For any $\nu\in \mc{V}$ which is not a contracted curve valuation, set $f_\bullet\nu := c(f,\nu)^{-1}f_*\nu\in\mc{V}$.
It can be shown that the map $f_\bullet\colon  \mc{V}\dashrightarrow \mc{V}$, defined away from contracted curve valuations, extends uniquely to an everywhere defined weakly continuous function $f_\bullet \colon \mc{V}\to \mc{V}$, see \cite[Theorem 3.1]{favre-jonsson:eigenval}.
The map $f_\bullet$ sends divisorial, irrational, and infinitely singular valuations to valuations of the same type, and sends non-contracted curve valuations to curve valuations.
The image of a contracted curve valuation is divisorial.
Finally, $f_\bullet$ is functorial in the sense that $(f^n)_\bullet = (f_\bullet)^n$.
We will heavily make use of the following three properties of $f_\bullet$.


\begin{prop}\label{prop:regularity}
Let $f\colon (\C^2,0)\to (\C^2,0)$ be a dominant superattracting fixed point germ.
\begin{enumerate}
\item[$1.$] If $T$ is the tree from \hyperref[prop:loc_const]{Proposition~\ref*{prop:loc_const}}, then $f_\bullet$ is order preserving on the components of $\mc{V}\smallsetminus T$.
\item[$2.$] \label{jacobian_formula}If $\nu\in \mc{V}$ has finite thinness, then so does $f_\bullet\nu$, and in fact one has the \emph{Jacobian formula}
\[c(f,\nu)A(f_\bullet\nu) = A(\nu) + \nu(J_f),\]
where $J_f$ is the Jacobian determinant of $f$.
\item[$3.$] For any $\nu\in \mc{V}$ with $\nu\neq \ord_0$, there exist valuations $\ord_0 = \nu_0<\nu_1<\cdots<\nu_n = \nu$ such that $f_\bullet$ maps each segment $I_j = [\nu_{j-1}, \nu_j]$ monotonically onto its image.
Moreover, the $\nu_j$ can be chosen so that $\nu_j$ is divisorial for $j = 1,\ldots, n-1$ and so that for each $j$ there exist integers $a,b,c,d\in \N$ with $ad - bc\neq 0$ for which
\[
\alpha(f_\bullet \nu) = \frac{a \alpha(\nu) + b}{c\alpha(\nu) + d}
\]
whenever $\nu\in I_j$.
\end{enumerate}
\end{prop}

Statement (1) in this proposition is immediate from \hyperref[prop:loc_const]{Proposition~\ref*{prop:loc_const}}.
Statement (2), on the other hand, is not immediately obvious from \hyperref[def:skewthin]{Definition~\ref*{def:skewthin}}.
However, from an alternate, more geometric, definition of thinness it is quite easy to derive, see for instance \cite[Lemma 4.7]{jonsson:berkovich} or \cite[Proposition 1.9]{favre:holoselfmapssingratsurf}.
As a consequence of the \hyperref[jacobian_formula]{Jacobian formula}, we see that $f_\bullet$ maps the subtree $\mc{V}_A$ of finite thinness valuations into itself.
Finally, statement (3) is proved in \cite[Theorem 3.1]{favre-jonsson:eigenval}.

\begin{rem}\label{rem:JacCharp}
The \hyperref[jacobian_formula]{Jacobian formula} fails when $\C$ is replaced by a field $k$ of characteristic $p>0$.
As an example, consider the map $f(x,y) = (x^p,y^p)$. One has $f_\bullet\ord_0 = \ord_0$, so the Jacobian formula should yield $2p = 2 + \ord_0(J_f)$.
Since $J_f = 0$, however, $\ord_0(J_f) = +\infty$.
Even if we stipulate that $J_f\neq 0$, the Jacobian formula can fail.
For instance, if $f(x,y) = (x^p(1+x), y^p(1+y))$, then once again $f_\bullet \ord_0 = \ord_0$, and the Jacobian formula should yield $2p = 2 + \ord_0(J_f) = 2 + \ord_0(x^py^p) = 2+2p$.
\end{rem}

Suppose that $f_\bullet\nu = \mu$.
Then $f_\bullet$ induces a map from the tangent space of $\nu$ into the tangent space of $\mu$, which we denote by $df_\bullet$.
If $\vec{v}$ is a tangent vector at $\nu$, then $df_\bullet\vec{v}$ is the tangent vector at $\mu$ defined as follows.
Let $[\nu, \nu']$ be any small interval in the direction $\vec{v}$.
Then $f_\bullet([\nu, \nu'])$ is a subinterval $[\mu,\mu']$ in some tangent direction $\vec{u}$ at $\mu$.
We set $df_\bullet\vec{v} = \vec{u}$.
If $\nu$ is not a contracted curve semivaluation, then the tangent map $df_\bullet$ surjects onto the tangent space of $\mu$.
The most interesting case is when $\nu$ and $\mu$ are divisorial valuations, say with $\nu = b_E^{-1}\ord_E$ for some exceptional prime $E$ of a blowup $\pi\colon X\to \C^2$ and $\mu = b_F^{-1}\ord_F$ for some exceptional prime $F$ of a blowup $\pi'\colon X'\to \C^2$.
In this case, the tangent spaces of $\nu$ and $\mu$ are in bijective correspondence with the points of $E$ and $F$, respectively. 
The map $f\colon X\dashrightarrow X'$ sends $E$ onto $F$, and the tangent map $df_\bullet$ is given by $f\colon E\to F$.

It is easy to see that the sequence $c(f^n, \nu)$ of attraction rates along a fixed valuation $\nu\in \mc{V}_A$ is a \emph{multiplicative dynamical cocycle}, i.e., that 
\[
c(f^n, \nu) = \prod_{i=0}^{n-1} c(f, f_\bullet^i\nu).
\]
It follows that the limit $c_\infty(\nu) := \lim c(f^n, \nu)^{1/n}$ exists.
In \cite[\S8]{jonsson:berkovich} it is shown that the value $c_\infty(\nu)$ is independent of $\nu\in \mc{V}_A$ and is a quadratic integer $c_\infty>1$.
In particular, if $\nu\in \mc{V}_A$ is fixed by $f_\bullet$, then $c(f,\nu) = c_\infty$.
The proof that $c_\infty$ is a quadratic integer makes essential use of the following fixed point theorem for $f_\bullet$.

\begin{thm}[{\cite[Theorem 4.2]{favre-jonsson:eigenval}}]\label{thm:eigenval}
There is a $\nu_\star\in \mc{V}$ such that $f_\bullet\nu_\star = \nu_\star$ and $c(f,\nu_\star) = c_\infty$.
Moreover, the fixed point $\nu_\star$ is either quasimonomial, or an attracting end, that is, there exists an $f_\bullet$-invariant segment $I=[\nu, \nu_\star]$ such that $f_\bullet^n(I) \to \nu_\star$.
\end{thm}
\begin{defi}
Any valuation $\nu_\star$ satisfying the conditions of \hyperref[thm:eigenval]{Theorem~\ref*{thm:eigenval}} is called an \emph{eigenvaluation} of $f$.
\end{defi}

While this theorem is powerful enough to conclude $c_\infty$ is a quadratic integer, in order to prove that the sequence $c(f^n)$ satisfies a linear integral recursion relation, we will need to study in more detail how attracting the eigenvaluation $\nu_\star$ is.
In the next two sections, we will show that \emph{every} finite thinness valuation $\nu\in \mc{V}_A$ is attracted to some eigenvaluation $\nu_\star$.
The main tool in doing so will be an equicontinuity theorem for $f_\bullet$, which we spend the remainder of the section proving.

Recall that if $a\colon \mc{V}\to [-\infty, \infty]$ is a parameterization, then $\mc{V}_a$ denotes the set of $\nu\in \mc{V}$ with $|a(\nu)|<\infty$, and $d_\alpha$ is the metric on $\mc{V}_a$ induced by $a$.
We will work now with two parameterizations on $\mc{V}$, namely the thinness parameterization $A$ and its reciprocal	$1/A$.
Note that $\mc{V}_{1/A} = \mc{V}$.

\begin{lem}\label{lem:potentialestimate}
Let $\nu, \mu\in \mc{V}_A$.
Then $|\nu(\phi) - \mu(\phi)|\leq \ord_0(\phi)d_A(\nu,\mu)$ for each function $\phi\in \mf{m}$. 
\end{lem}
\begin{proof} Note that it suffices to prove the lemma for irreducible $\phi$.
First suppose that $\nu$ and $\mu$ are comparable, with say $\nu<\mu$.
By \hyperref[prop:skewness_eval]{Proposition~\ref*{prop:skewness_eval}}, one has that
\[
\mu(\phi) - \nu(\phi) = \ord_0(\phi)[\alpha(\mu\wedge\nu_\phi) - \alpha(\nu\wedge \nu_\phi)] \leq \ord_0(\phi)d_\alpha(\nu,\mu).
\]
It is immediate from the definition of thinness in \hyperref[def:skewthin]{Definition~\ref*{def:skewthin}} that $d_\alpha(\nu,\mu)\leq d_A(\nu,\mu)$, and thus $\mu(\phi) - \nu(\phi)\leq \ord_0(\phi)d_A(\nu,\mu)$, as desired.
Now assume that $\nu$ and $\mu$ are general.
Then
\begin{align*}
|\nu(\phi) - \mu(\phi)| &\leq \max\{\nu(\phi), \mu(\phi)\} - (\nu\wedge\mu)(\phi)\leq \ord_0(\phi)\max\{d_A(\nu,\nu\wedge\mu), d_A(\mu, \nu\wedge\mu)\}\\
& \leq \ord_0(\phi)d_A(\nu,\mu). \qedhere
\end{align*}
\end{proof}

\begin{cor}\label{cor:thinnessestimate}
Let $\nu,\mu\in \mc{V}_A$.
Then one has the inequality
\[
|c(f,\nu)A(f_\bullet\nu) - c(f,\mu)A(f_\bullet\mu)|\leq d_A(\nu,\mu)[1 + \ord_0(J_f)].
\]
\end{cor}
\begin{proof}
Using the \hyperref[jacobian_formula]{Jacobian formula}, the left hand side of this inequality is exactly
\[
|A(\nu) - A(\mu) + \nu(J_f) - \mu(J_f)|\leq d_A(\nu,\mu) + |\nu(J_f) - \mu(J_f)|.
\]
From \hyperref[lem:potentialestimate]{Lemma~\ref*{lem:potentialestimate}} we know that $|\nu(J_f) - \mu(J_f)|\leq \ord_0(J_f)d_A(\nu,\mu)$, from which the corollary easily follows.
\end{proof}

\begin{thm}[Equicontinuity]\label{thm:equicontinuity}
Let $\nu,\mu\in \mc{V}_A$.
Then $d_{1/A}(f_\bullet\nu, f_\bullet\mu) \leq 2^{-1}d_A(\nu,\mu)$. 
\end{thm}
\begin{proof}
By \hyperref[prop:regularity]{Proposition~\ref*{prop:regularity}}, the interval $[\nu,\mu]$ can be decomposed into finitely many abutting closed subintervals on which $f_\bullet$ is monotonic.
Thus, with no loss of generality we can assume that $\mu<\nu$ and that $f_\bullet$ is order preserving or reversing on $[\mu,\nu]$.
Then, by definition,
\begin{equation}\label{eqn:recip}
d_{1/A}(f_\bullet\nu, f_\bullet\mu) = |A(f_\bullet\nu)^{-1} - A(f_\bullet\mu)^{-1}| = \frac{|A(f_\bullet\nu) - A(f_\bullet\mu)|}{A(f_\bullet\nu)A(f_\bullet\mu)}.
\end{equation}

\textbf{Case 1:}
Assume first that $f_\bullet$ is order preserving on $[\mu,\nu]$.
Since $c(f,\nu)\geq c(f,\mu)$ one has
\[
c(f,\mu)[A(f_\bullet\nu) - A(f_\bullet\mu)] \leq c(f,\nu)A(f_\bullet \nu) - c(f,\mu)A(f_\bullet\mu).
\]
Thus by \hyperref[cor:thinnessestimate]{Corollary~\ref*{cor:thinnessestimate}}, we see
\[
A(f_\bullet\nu) - A(f_\bullet\mu) \leq \frac{d_A(\nu,\mu)}{c(f,\mu)}[1 + \ord_0(J_f)].
\]
Applying this inequality to \hyperref[eqn:recip]{Equation~(\ref*{eqn:recip})} then yields
\[
d_{1/A}(f_\bullet\nu, f_\bullet\mu) \leq \frac{d_A(\nu,\mu)[1 + \ord_0(J_f)]}{c(f,\mu)A(f_\bullet\mu)A(f_\bullet\nu)} = \frac{d_A(\nu,\mu)[1 + \ord_0(J_f)]}{A(f_\bullet\nu)[A(\mu) + \mu(J_f)]}\leq \frac{d_A(\nu,\mu)}{A(f_\bullet\nu)}\leq \frac{d_A(\nu,\mu)}{2}.
\]

\textbf{Case 2:}
We assume now that $f_\bullet$ is order reversing on $[\mu,\nu]$.
By the \hyperref[jacobian_formula]{Jacobian formula},
\begin{align*}
A(f_\bullet\mu) - A(f_\bullet\nu) & = \frac{A(\mu) + \mu(J_f)}{c(f,\mu)} - \frac{A(\nu) + \nu(J_f)}{c(f,\nu)}\\
& = \frac{A(\mu) + \mu(J_f) - A(\nu) - \nu(J_f)}{c(f,\mu)} + (A(\nu)+ \nu(J_f))\left[\frac{1}{c(f,\mu)} - \frac{1}{c(f,\nu)}\right].
\end{align*}
Since $\mu<\nu$, the expression $A(\mu) + \mu(J_f) - A(\nu) - \nu(J_f)$ is non-positive, and thus
\[
A(f_\bullet\mu) - A(f_\bullet\nu) \leq (A(\nu)+ \nu(J_f))\left[\frac{1}{c(f,\mu)} - \frac{1}{c(f,\nu)}\right] =A(f_\bullet \nu)\left[\frac{c(f,\nu) - c(f,\mu)}{c(f,\mu)}\right].
\]
Putting this inequality into \hyperref[eqn:recip]{Equation~(\ref*{eqn:recip})} yields
\[
d_{1/A}(f_\bullet\nu, f_\bullet\mu) \leq \frac{1}{A(f_\bullet\mu)}\left[\frac{c(f,\nu) - c(f,\mu)}{c(f,\mu)}\right].
\]
If $i \in\{1,2\}$ is such that $c(f,\mu) = \mu(f_i)$, then \hyperref[lem:potentialestimate]{Lemma~\ref*{lem:potentialestimate}} gives the estimate
\[
c(f,\nu) - c(f,\mu) \leq \nu(f_i) - \mu(f_i)\leq \ord_0(f_i)d_A(\nu,\mu)\leq c(f,\mu)d_A(\nu,\mu),
\]
and thus we have
\[
d_{1/A}(f_\bullet\nu, f_\bullet\mu)\leq \frac{d_A(\nu,\mu)}{A(f_\bullet \mu)}\leq \frac{d_A(\nu,\mu)}{2}.
\]
This completes the proof.
\end{proof}

\section{The case of a single eigenvaluation}

Fix a dominant superattracting holomorphic fixed point germ $f\colon (\C^2,0)\to (\C^2,0)$.
In this and the next section we study the basins of attraction for eigenvaluations of $f$, expanding on the work of \cite{favre-jonsson:eigenval}.
In that work, Favre-Jonsson show the existence of eigenvaluations and prove that they are always in some sense locally attracting.
We will show that they are globally attracting.
Specifically, we will prove the following theorem.

\begin{thm}\label{thm:basins}
Let $f\colon (\C^2,0)\to (\C^2,0)$ be a dominant superattracting holomorphic germ.
Then we are in one of the following three situations.
\begin{enumerate}
\item[$1.$] There is a unique eigenvaluation $\nu_\star$ of $f$, which is an end of $\mc{V}$.
In this case, for every $\nu\in \mc{V}$ with finite thinness one has $f^n_\bullet\nu\to \nu_\star$ in the weak topology as $n\to \infty$.
\item[$2.$] There is a unique eigenvaluation $\nu_\star$ of $f$, which is quasimonomial, such that for every $\nu\in \mc{V}$ with finite thinness one has $f^n_\bullet\nu\to \nu_\star$ with respect to the thinness metric.
\item[$3.$] There is a non-degenerate segment $I\subset\mc{V}$ of fixed points of $f_\bullet^2$ such that for every $\nu\in \mc{V}$ of finite thinness there is a $\nu_\star \in I$ such that $f^{2n}_\bullet\nu\to \nu_\star$ with respect to the thinness metric.
The valuation $\nu_\star$ is given by $r(\nu)$, where $r\colon\mc{V} \to I$ is the retraction map to $I$.
\end{enumerate}
\end{thm}

In this section we will focus on the first two cases of \hyperref[thm:basins]{Theorem~\ref*{thm:basins}}; in \S4 we will study the third case.
\hyperref[thm:basins]{Theorem~\ref*{thm:basins}} is the main ingredient in the proofs of \hyperref[thm:recursionrelation]{Theorem~\ref*{thm:recursionrelation}} and \hyperref[thm:weakstability]{Theorem~\ref*{thm:weakstability}}, so these two sections are the heart of this article.

\subsection{The case of an end eigenvaluation}
Suppose that $\nu_\star$ is an eigenvaluation of $f$ which is an end of $\mc{V}$, that is, $\nu_\star$ is either a curve valuation or an infinitely singular valuation.
Consider the set
\[
B = \{\nu\in \mc{V}_A : f^n_\bullet \nu\to \nu_\star \mbox{ weakly as }n\to \infty\}.
\]
We will show that $B = \mc{V}_A$ using a connectedness argument, that is, by proving that $B$ is a nonempty open and closed set with respect to the thinness topology.

The fact that $B$ is nonempty and open is a consequence of \cite[Proposition 5.2(i)]{favre-jonsson:eigenval}.
In this proposition it is shown that for any sufficiently large $\mu<\nu_\star$, the weak open set $U = \{\nu\in \mc{V} : \nu>\mu\}$ is such that $f_\bullet U\Subset U$ and $\mc{V}_A\cap U\subseteq B$.
It follows easily that
\[
B = \mc{V}_A\cap \bigcup_{n\geq 0} f_\bullet^{-n}(U)
\]
is weakly open.
Since the thinness topology is strictly stronger than the weak topology, $B$ is open in the thinness topology as well.
It remains to show that $B$ is closed in the thinness topology; we do this using the equicontinuity \hyperref[thm:equicontinuity]{theorem~\ref*{thm:equicontinuity}}.

\begin{prop}
The basin $B$ is closed in the thinness topology, and hence $B = \mc{V}_A$.
\end{prop}
\begin{proof}
Let $\nu_m$ be a sequence in $B$ which converges in the thinness topology to a valuation $\nu_0\in \mc{V}_A$.
Let $\mu<\nu_\star$ be large enough that $U := \{\nu\in \mc{V} : \nu>\mu\}$ satisfies $f_\bullet U\Subset U$ and $\mc{V}_A\cap U\subseteq B$.
Let $\eps>0$ be small enough that the open $d_{1/A}$-ball $B_{1/A}(f_\bullet\mu, \eps)$ is contained in $U$.
For $m$ chosen large enough that $d_A(\nu_0,\nu_m)<\eps$, the equicontinuity \hyperref[thm:equicontinuity]{theorem~\ref*{thm:equicontinuity}} implies $d_{1/A}(f^n_\bullet\nu_m, f^n_\bullet \nu_0)<\eps/2$ for all $n$.
Thus if $n$ is large enough that $f_\bullet^n\nu_m\in f_\bullet U$, one has $f^n_\bullet\nu_0\in U$.
Since $\mc{V}_A\cap U\subseteq B$, we conclude that $\nu_0\in B$.
\end{proof}

\begin{cor}
If $f$ has an end eigenvaluation $\nu_\star$, then $f$ falls into situation 1 of \hyperref[thm:basins]{Theorem~\ref*{thm:basins}}.
\end{cor}
\begin{proof}
We only must show $\nu_\star$ is the unique eigenvaluation of $f$.
Note that $f$ cannot have a quasimonomial eigenvaluation, since the orbit of every quasimonomial eigenvaluation is weakly attracted to $\nu_\star$.
On the other hand, $f$ cannot have another end eigenvaluation $\mu_\star$ since then the orbit of any finite thinness valuation would be weakly attracted to both $\nu_\star$ and $\mu_\star$, which is impossible.
\end{proof}

\subsection{The case of an irrational eigenvaluation}
Suppose now that $f$ has an irrational eigenvaluation $\nu_\star$.
We now consider the basin
\[
B = \{\nu\in \mc{V}_A : f^n_\bullet\nu \to \nu_\star \mbox{ in the thinness topology as } n\to \infty\}.
\]
We would like to argue in a similar way to how we argued in the case of an end eigenvaluation that $B = \mc{V}_A$.
As we will see, however, it is not always the case that $B = \mc{V}_A$.
On the other hand, it \emph{is} always the case that $B$ is closed, as we will prove momentarily.
The proof we give for this is valid not just when $\nu_\star$ is irrational, but when it has finite thinness. 

\begin{lem}\label{lem:compare}
Suppose $\nu_\star$ is an eigenvaluation of finite thinness.
Then there exist $C,\delta>0$ depending only on $A(\nu_\star)$ such that for all $\nu\in \mc{V}_A$ with $d_{1/A}(\nu,\nu_\star)<\delta$ one has $d_{A}(\nu, \nu_\star)\leq C d_{1/A}(\nu,\nu_\star)$.
\end{lem}
\begin{proof}
We take $C = 3A(\nu_\star)^2$, and choose $\delta>0$ to be any value small enough that
\[
A(\nu_\star)^2<\frac{A(\nu_\star)^2}{1 - \delta A(\nu_\star)}<2A(\nu_\star)^2.
\]
Suppose that $\nu \in \mc{V}_A$ is such that $d_{1/A}(\nu,\nu_\star)<\delta$.
Let $\mu = \nu \wedge \nu_\star$.
Then \begin{equation}\label{eqn:lower}
d_{1/A}(\mu,\nu_\star) = \frac{A(\nu_\star) - A(\mu)}{A(\mu)A(\nu_\star)}\geq \frac{A(\nu_\star) - A(\mu)}{A(\nu_\star)^2}.
\end{equation}
Similarly, one has
\[
d_{1/A}(\mu,\nu) = \frac{A(\nu) - A(\mu)}{A(\mu)A(\nu)} = \frac{A(\mu)}{A(\nu)}\cdot\frac{A(\nu) - A(\mu)}{A(\mu)^2} = [1-d_{1/A}(\mu, \nu)A(\mu)]\frac{A(\nu) - A(\mu)}{A(\mu)^2}.
\]
Rearranging this expression gives that
\begin{equation}\label{eqn:upper}
A(\nu) - A(\mu) = \frac{A(\mu)^2 d_{1/A}(\mu,\nu)}{1 - d_{1/A}(\mu,\nu)A(\mu)} \leq \frac{A(\nu_\star)^2 d_{1/A}(\mu,\nu)}{1 - d_{1/A}(\nu_\star,\nu)A(\nu_\star)}\leq 2A(\nu_\star)^2d_{1/A}(\mu,\nu),
\end{equation}
where the last inequality follows from our choice of $\delta$.
Combining \hyperref[eqn:lower]{Equation~(\ref*{eqn:lower})} and \hyperref[eqn:upper]{Equation~(\ref*{eqn:upper})} then yields the desired inequality $d_A(\nu,\nu_\star) \leq 3 A(\nu_\star)^2d_{1/A}(\nu,\nu_\star)$.
\end{proof}

\begin{prop}\label{prop:basin_closed}
Suppose $\nu_\star$ is an eigenvaluation of finite thinness, and let $B$ be the set of $\nu\in \mc{V}_A$ whose orbit converges to $\nu_\star$ in the thinness topology.
Then $B$ is closed in the thinness topology.
\end{prop}
\begin{proof}
Let $C$ and $\delta$ be as given in \hyperref[lem:compare]{Lemma~\ref*{lem:compare}}.
Suppose that $\nu_m$ is a sequence in $B$ converging in the thinness metric to some $\nu_0\in \mc{V}_A$.
Let $0<\eps<\delta$ be given, and choose an $m$ large enough so that $d_A(\nu_m, \nu_0)<\eps$.
One has $d_A(f^n_\bullet\nu_m, \nu_\star)<\eps$ for all sufficiently large $n$, since $\nu_m\in B$.
By the equicontinuity \hyperref[thm:equicontinuity]{theorem~\ref*{thm:equicontinuity}}, $d_{1/A}(f^n_\bullet \nu_m, \nu_\star)<\eps/2$ for large $n$.
But then
\[
d_{1/A}(f^n_\bullet\nu_0,\nu_\star)\leq d_{1/A}(f^n_\bullet\nu_0, f^n_\bullet\nu_m) + d_{1/A}(f^n_\bullet \nu_m, \nu_\star)\leq \frac{1}{2}d_A(\nu_0,\nu_m) + \frac{\eps}{2} < \eps<\delta.
\]
Thus $d_A(f^n_\bullet\nu_0, \nu_\star)<C\eps$ for large $n$.
\end{proof}

We now return to the assumption that $\nu_\star$ is an irrational eigenvaluation for $f$.
We have shown that the basin $B$ is closed in the thinness topology, and we wish to understand when it is open.
The next proposition describes when this happens.

\begin{prop}
There exist divisorial valuations $\nu_1<\nu_\star<\nu_2$ which can be taken as close to $\nu_\star$ as desired such that the interval $I = [\nu_1,\nu_2]$ is $f_\bullet$-invariant.
Moreover, either
\begin{enumerate}
\item[$1.$] $f^n_\bullet \nu\to \nu_\star$ in the thinness topology for all $\nu\in I$, or
\item[$2.$] $f^2_\bullet$ fixes every point of $I$.
\end{enumerate}
In the first case, $f$ falls into situation 2 of \hyperref[thm:basins]{Theorem~\ref*{thm:basins}}.
\end{prop}
\begin{proof}
Except for the last statement, this is exactly \cite[Proposition 5.2(iii)]{favre-jonsson:eigenval}.
Thus we only must show that in case 1, $f$ falls into situation 2 of \hyperref[thm:basins]{Theorem~\ref*{thm:basins}}.
Assume, then, that we are in case 1.
Let $\vec{v}_i$ be the tangent vector at $\nu_i$ in the direction of $\nu_\star$ for $i = 1,2$, and set $U = U(\vec{v}_1)\cap U(\vec{v}_2)$.
Then $U$ is a weak open set, and by \cite[Proposition 5.2(iii)]{favre-jonsson:eigenval} it is $f_\bullet$-invariant.

Let $T$ be the set of $\nu\in \mc{V}_A$ such that at least one of the functions $\nu\mapsto c(f,\nu)$ and $\nu\mapsto \nu(J_f)$ is not locally constant at $\nu$.
By \hyperref[prop:loc_const]{Propositions~\ref*{prop:loc_const}} and \ref{prop:skewness_eval}, $T$ is a finite subtree of $\mc{V}_A$.
Therefore, by choosing the $\nu_i$ close enough to $\nu_\star$, we can arrange that $U\cap T\subseteq (\nu_1,\nu_2)$.
Let $r\colon U\to I$ be the retraction map $r\nu :=\nu\wedge\nu_2$.
By \hyperref[prop:regularity]{Proposition~\ref*{prop:regularity}(1)}, $f_\bullet$ is order preserving on $U\smallsetminus I$, which says precisely that $f_\bullet \circ  r= r\circ f_\bullet$ on $U$.
Similarly,  the fact that $c(f,-)$ is constant on the components of $U\smallsetminus I$ says precisely that $c(f,\nu) = c(f,r\nu)$ for all $\nu\in U$.
For any $\nu\in U$, one then has
\[
d_A(f^n_\bullet\nu, \nu_\star) \leq d_A(f^n_\bullet\nu, rf^n_\bullet\nu) + d_A(rf^n_\bullet\nu, \nu_\star) = d_A(f^n_\bullet\nu, f^n_\bullet r\nu) + d_A(f^n_\bullet r\nu, \nu_\star).
\]
The second term on the right hand side of this expression tends to $0$ as $n\to \infty$ by hypothesis: we are assuming to be in case 1.
By the \hyperref[jacobian_formula]{Jacobian formula}, the first term can be computed to be
\[
d_A(f^n_\bullet\nu, f^n_\bullet r\nu) = c(f^n, r\nu)^{-1}d_A(\nu, r\nu)\to 0\,\,\,\,\,\mbox{ as }n\to \infty.
\]
Therefore $d_A(f^n_\bullet\nu, \nu_\star)\to 0$, that is, $U\subseteq B$.
We conclude that $B$ is open, and thus by connectedness $B = \mc{V}_A$.
Note that $f$ cannot have another quasimonomial eigenvaluation, since all quasimonomial valuations are attracted to $\nu_\star$.
Also, $f$ cannot have an end eigenvaluation, since end eigenvaluations attract every finite thinness valuation.
This proves $\nu_\star$ is unique, and hence we are in situation 2 of \hyperref[thm:basins]{Theorem~\ref*{thm:basins}}.
\end{proof}

\subsection{The case of a divisorial eigenvaluation}
Suppose now that $\nu_\star$ is a divisorial eigenvaluation of $f$.
The analysis in this case is more delicate, owing largely to the fact that the tangent space at $\nu_\star$ is more complicated.
Again, let
\[
B = \{\nu\in \mc{V}_A : f^n_\bullet\nu\to \nu_\star \mbox{ in the thinness topology as }n\to \infty\}.
\]
We have seen that $B$ is closed in the thinness topology, and we wish to determine when it is open.
The next proposition describes when this happens.

\begin{prop}\label{prop:divisorial}
If $\nu_\star$ is a divisorial eigenvaluation for $f$, then either
\begin{enumerate}
\item[$1.$] $f^n_\bullet \nu\to \nu_\star$ in the thinness topology for all $\nu\in \mc{V}_A$, or
\item[$2.$] there is a segment $I = [\nu_\star, \nu]$ and an $m\geq 1$ such that $f^m_\bullet$ acts as the identity on $I$.
\end{enumerate}
In the first case, $f$ falls into situation 2 of \hyperref[thm:basins]{Theorem~\ref*{thm:basins}}.
\end{prop}
\begin{proof}
Assume that we are not in case 2, so that no such interval $I$ exists.
Let $T\subset \mc{V}_A$ be the finite subtree consisting of those valuations $\nu\in \mc{V}_A$ at which at least one of the functions $\nu\mapsto c(f,\nu)$ and $\nu\mapsto \nu(J_f)$ is not locally constant.
Since $T$ is a finite tree, there can be at most finitely many tangent directions at $\nu_\star$ which meet $T$.
Label these directions $\vec{v}_1, \ldots, \vec{v}_k$.
We will need the following two lemmas, the first of which is elementary and left to the reader.

\begin{lem}\label{lem:mobius}
Let $H(t) = (at+b)/(ct+d)$ be a M\"{o}bius transformation with $a,b,c,d\in \N$.
Suppose that $t_0>0$ is a fixed point of $H$, and that $H$ is increasing near $t_0$.
Then $H'(t_0)\leq 1$, with equality if and only if $H(t) = t$.
\end{lem}

\begin{lem}\label{lem:directed_basins}
There exists a function $\rho$ from the tangent space at $\nu_\star$ to $(0,+\infty]$ such that
\begin{enumerate}
\item[$1.$] if $\nu\in \mc{V}_A$ lies in the direction $\vec{v}$ and $d_A(\nu, \nu_\star)<\rho(\vec{v})$, then $\nu\in B$, and
\item[$2.$] one has $\inf_{\vec{v}} \rho(\vec{v}) = \min_{i=1,\ldots, k}\rho(\vec{v}_i)>0$.
\end{enumerate}
\end{lem}
\begin{proof}
Let $S_1$ denote the set of tangent vectors $\vec{v}$ at $\nu_\star$ such that $df^n_\bullet\vec{v}\notin\{\vec{v}_1,\ldots, \vec{v}_k\}$ for all $n\geq 1$, and let $S_2$ denote the set of those $\vec{v}_1,\ldots, \vec{v}_k$ which are $df_\bullet$-periodic.
Then $S_1$ and $S_2$ are disjoint invariant sets, and the orbit of any tangent vector $\vec{v}$ at $\nu_\star$ eventually lands in either $S_1$ or $S_2$.
We will begin by defining $\rho(\vec{v})$ for those $\vec{v}$ lying in either $S_1$ or $S_2$.

First suppose $\vec{v}\in S_1$.
We then set $\rho(\vec{v}) = +\infty$.
By \hyperref[prop:regularity]{Proposition~\ref*{prop:regularity}(1)}, $f_\bullet$ is order preserving on $U(\vec{v})$, so that $f_\bullet(U(\vec{v}))\subseteq U(df_\bullet\vec{v})$.
Moreover, since $c(f,\nu) = c_\infty$ and $\nu(J_f) = \nu_\star(J_f)$ for every $\nu\in U(\vec{v})$, the \hyperref[jacobian_formula]{Jacobian formula} gives $d_A(f_\bullet\nu, \nu_\star) = c_\infty^{-1}d_A(\nu, \nu_\star)$ for $\nu\in U(\vec{v})$.
Upon iterating, it follows that $d_A(f^n_\bullet\nu, \nu_\star) = c_\infty^{-n}d_A(\nu, \nu_\star)$, so that $U(\vec{v})\subseteq B$.

Now suppose that $\vec{v}\in S_2$, and let $m$ be the period of $\vec{v}$.
Let $I = [\nu_\star, \mu]$ be some segment in the direction $\vec{v}$.
By \hyperref[prop:regularity]{Proposition~\ref*{prop:regularity}(3)}, there is a subinterval $J\subseteq I$ containing $\nu_\star$ such that $f^m_\bullet J\subseteq I$ and such that there is a M\"{o}bius transformation $H(t) = (at+b)/(ct+d)$ with $a,b,c,d\in \N$ for which $\alpha(f^m_\bullet \nu) = H(\alpha(\nu))$ for $\nu\in J$.
By \hyperref[lem:mobius]{Lemma~\ref*{lem:mobius}}, either $H'(\alpha(\nu_\star)) <1$ or else $H(t) = t$.
However, since we have assumed that $f^m_\bullet$ cannot act as the identity on $J$, we must have $H'(\alpha(\nu_\star))<1$.
Thus $\nu_\star$ is attracting in the sense that there is a subinterval $J'\subseteq J$ containing $\nu_\star$ such that $f^{mn}_\bullet \to \nu_\star$ on $J'$ in the skewness topology.
However, the thinness and skewness topologies are equivalent on $J'$, so in fact $J'\subseteq B$.

Suppose that $J' = [\nu_\star, \nu_0]$, and let $\vec{w}$ be the tangent vector at $\nu_0$ in the direction of $\nu_\star$.
Define $U = U(\vec{v})\cap U(\vec{w})$.
Let $r\colon U\to J'$ be the usual retraction map, mapping $\nu\in U$ to the closest point of $J'$ to $\nu$.
If $J'$ is chosen small enough, then the functions $\nu\mapsto c(f^m,\nu)$ and $\nu\mapsto \nu(J_{f^m})$ are constant on components of $U\smallsetminus J'$.
It follows that $f^m_\bullet$ is order preserving on components of $U\smallsetminus J'$.
This says precisely that $U$ is $f^m_\bullet$-invariant, and that $r\circ f^m_\bullet = f^m_\bullet\circ r$ on $U$.
Using the fact that $c(f^m,\nu)$ and $\nu(J_{f^m})$ are constant on components of $U\smallsetminus J'$ one derives
\[
d_A(f^{mn}_\bullet\nu, \nu_\star) \leq d_A(f^{mn}_\bullet \nu, rf^{mn}_\bullet\nu) + d_A(rf^{mn}_\bullet\nu, \nu_\star) \leq c(f^{mn}, r\nu)^{-1}d_A(\nu, r\nu) + d_A(f^{mn}_\bullet r\nu, \nu_\star)\to 0
\]
for $\nu\in \mc{V}_A\cap U$, so that $\mc{V}_A\cap U\subseteq B$.
We may therefore define $\rho(\vec{v}) = d_A(\nu_\star, \nu_0)$.

We have defined $\rho$ on $S_1\cup S_2$ in such a way that properties (1) and (2) hold.
We now inductively define $\rho$ on all tangent vectors as follows.
Suppose that $\vec{v}$ is such that $\rho(df_\bullet\vec{v})$ has been defined.
If $\vec{v}\notin \{\vec{v}_1,\ldots, \vec{v}_k\}$, then one has $f_\bullet U(\vec{v})\subseteq U(df_\bullet\vec{v})$ and $d_A(f_\bullet \nu, \nu_\star) = c_\infty^{-1}d_A(\nu, \nu_\star)$ for every $\nu\in \mc{V}_A\cap U(\vec{v})$.
Thus we may define $\rho(\vec{v}) = c_\infty \rho(df_\bullet \vec{v})$.
On the other hand, if $\vec{v}\in \{\vec{v}_1,\ldots, \vec{v}_k\}$, simply choose $\rho(\vec{v})$ to be a value small enough that every $\nu\in U(\vec{v})$ with $d_A(\nu, \nu_\star)<\rho(\vec{v})$ satisfies $f_\bullet \nu\in U(df_\bullet\vec{v})$ and $d_A(f_\bullet\nu, \nu_\star)<\rho(df_\bullet \vec{v})$.
\end{proof}

We now continue with the proof of \hyperref[prop:divisorial]{Proposition~\ref*{prop:divisorial}}.
Let $\rho_0 = \inf_{\vec{v}}\rho(\vec{v})$, where $\rho$ is as in \hyperref[lem:directed_basins]{Lemma~\ref*{lem:directed_basins}}.
Then the open ball $B_A(\nu_\star, \rho_0)$ is contained in $B$.
It follows immediately that $B$ is open in the thinness topology, and hence $B = \mc{V}_A$.
As we have reasoned before, $\nu_\star$ is necessarily unique, putting us in situation 2 of \hyperref[thm:basins]{Theorem~\ref*{thm:basins}}.
\end{proof}

\section{The case of a segment of eigenvaluations}

In this section we complete the proof of \hyperref[thm:basins]{Theorem~\ref*{thm:basins}}, focusing on its third situation.
We therefore assume that $f\colon (\C^2,0)\to (\C^2,0)$ is a dominant superattracting holomorphic germ that does not fall into either situation 1 or 2 of \hyperref[thm:basins]{Theorem~\ref*{thm:basins}}.
Recall that in \S3 we proved that this implies there is a non-degenerate segment $J\subseteq \mc{V}$ of quasimonomial valuations such that
\begin{enumerate}
\item[$1.$] $J$ contains at least one eigenvaluation for $f$, and 
\item[$2.$] there is an integer $m\geq 1$ such that $J$ consists entirely of eigenvaluations for $f^m$.
\end{enumerate}
Moreover, in the case where $J$ contains an irrational eigenvaluation of $f$, we saw that one may take $m = 2$.
We will shortly see this remains true when $J$ contains a divisorial eigenvaluation of $J$.

\begin{prop}
The germ $f$ is necessarily finite, that is, $f$ has no contracted curves. 
\end{prop}
\begin{proof}
Since $f$ is finite if and only if $f^m$ is finite, we may without loss of generality assume that $J$ consists of eigenvaluations of $f$.
Moreover, by shrinking $J$ if necessary, we may assume $J = [\nu_0,\nu_1]$, where $\nu_0<\nu_1$.
We begin by noting that $c(f,-)\equiv c_\infty$ along $J$, since one always has $c(f,\nu) = c_\infty$ for eigenvaluations $\nu$.
If $\vec{v}$ is the tangent vector at $\nu_0$ in the direction of $\nu_1$, it follows by \hyperref[prop:loc_const]{Proposition~\ref*{prop:loc_const}} that $c(f,-)\equiv c_\infty$ on $U(\vec{v})$. 

Assume for contradiction that there is a contracted curve $C = \{\phi = 0\}$ for $f$.
Since $c(f,-)$ must be unbounded near the contracted curve valuation $\nu_\phi$, we can conclude $\nu_\phi\notin U(\vec{v})$.
Let $\psi\in \mf{m}$ be any irreducible element such that $\nu_1<\nu_\psi$.
Then by \hyperref[prop:skewness_eval]{Proposition~\ref*{prop:skewness_eval}}, $\nu(\psi) = \ord_0(\psi)\alpha(\nu)$ for every $\nu\in J$.
On the other hand, $f_*\nu = c_\infty\nu$ for $\nu\in J$, so that $\nu(\psi) = c_\infty^{-1}(f_*\nu)(\psi) = c_\infty^{-1}\nu(\psi\circ f)$.
Because $C$ is a contracted curve, $\phi\mid \psi\circ f$, and hence we can write
\[
\ord_0(\psi)\alpha(\nu) = \nu(\psi) = c_\infty^{-1}\nu(\phi) + c_\infty^{-1}\nu(\psi\circ f/\phi) = c_\infty^{-1}\ord_0(\phi)\alpha(\nu\wedge\nu_\phi) + c_\infty^{-1}\nu(\psi\circ f/\phi).
\]
Because $\nu_\phi\notin U(\vec{v})$, the term $\alpha(\nu\wedge\nu_\phi)$ is constant for $\nu\in J$.
Moreover, \hyperref[prop:skewness_eval]{Proposition~\ref*{prop:skewness_eval}} implies that $\nu(\psi\circ f/\phi)$ is a piecewise linear function with non-negative coefficients on $J$ with respect to the skewness parameterization.
We have thus derived that $\ord_0(\psi)\alpha(\nu)$ --- a linear function in $\alpha(\nu)$ with no constant term --- is equal the piecewise linear function $c_\infty^{-1}\ord_0(\phi)\alpha(\nu\wedge \nu_\phi) + c_\infty^{-1}\nu(\psi\circ f/\phi)$ with nonzero constant terms $\geq c_\infty^{-1}\ord_0(\phi)\alpha(\nu\wedge\nu_\phi)>0$, a contradiction.
\end{proof}

\begin{prop}\label{prop:tot_inv}
Every $\nu\in J$ is totally invariant for $f^m_\bullet$, that is, $f^{-m}_\bullet(\nu) = \{\nu\}$.
Moreover, if $\nu\in J$ is divisorial, then the degree of the tangent map $df_\bullet^m$ at $\nu$ is $>1$.
\end{prop}
\begin{proof}
Replacing $f$ by $f^m$, we may assume without loss of generality that $J$ consists entirely of eigenvaluations for $f$.
By the continuity of $f_\bullet$, we only need to prove that a dense set of $\nu\in J$ are totally invariant for $f_\bullet$.
We will prove it for the dense set $S$ of divisorial valuations $\nu\in J$ that are not the minimal element of $J$.

Let $\nu\in S$, and let $\mu\in S$ be such that $\mu<\nu$.
Let $\vec{v}$ be the tangent vector at $\mu$ in the direction of $\nu$.
Since $c(f,-)\equiv c_\infty$ on $J$, \hyperref[prop:loc_const]{Proposition~\ref*{prop:loc_const}} implies $c(f,-)$ is constant on $U(\vec{v})$, and thus by \hyperref[prop:regularity]{Proposition~\ref*{prop:regularity}(1)}, $f_\bullet$ is order preserving on $U(\vec{v})$.
It follows that if $\nu'\in \mc{V}$ is such that $f_\bullet\nu' = \nu$, then either $\nu' = \nu$ or $\nu'\wedge \nu<\nu$.
In order to proceed, we will need the geometric interpretation of skewness discussed in \S1.5.
Let $\pi\colon X\to \C^2$ be a blowup over $0$ such that $\nu = b_E^{-1}\ord_E$ and $\mu = b_F^{-1}\ord_F$ for some exceptional primes $E$ and $F$ of $\pi$.
Let $\pi'\colon X'\to \C^2$ be a blowup over $0$ that dominates $\pi$ such that $f$ lifts to a holomorphic map $f\colon X'\to X$ over $0$.
On the one hand, we have that $f_*Z_{\mu, \pi'} = Z_{f_*\mu,\pi} = c_\infty Z_{\mu,\pi}$ by \cite[Lemma 1.10]{favre:holoselfmapssingratsurf}, so that
\[
-(f_*Z_{\mu, \pi'}\cdot Z_{\nu,\pi}) = -c_\infty(Z_{\mu,\pi}\cdot Z_{\nu,\pi}) = c_\infty \alpha(\mu\wedge\nu) = c_\infty \alpha(\mu).
\]
On the other hand, because $f$ is finite we can apply the projection formula to deduce $c_\infty\alpha(\mu) = -(Z_{\mu,\pi'}\cdot f^*Z_{\nu,\pi})$.
Moreover, since $f$ is finite, \cite[Lemma 1.10]{favre:holoselfmapssingratsurf} shows that there exist positive constants $a_{\nu'}>0$ such that
\[
f^*Z_{\nu, \pi} = \sum_{f_\bullet\nu' = \nu} a_{\nu'}Z_{\nu', \pi'}
\]
(see also \cite[Proposition A.7]{boucksom-favre-jonsson:degreegrowthmeromorphicsurfmaps}).
We can then compute
\[
c_\infty\alpha(\mu) = -(Z_{\mu,\pi'}\cdot f^*Z_{\nu,\pi}) = \sum_{f_\bullet\nu' = \nu} -a_{\nu'}(Z_{\mu,\pi'}\cdot Z_{\nu',\pi'}) = \sum_{f_\bullet\nu' = \nu} a_{\nu'}\alpha(\mu\wedge \nu').
\]
If we take $\nu' = \nu$, then $\nu'\wedge \mu = \mu$.
On the other hand, if $\nu'\neq \nu$, then $\mu\wedge \nu' = \nu\wedge \nu'$ whenever $\mu$ is close enough to $\nu$.
Thus one has the equality
\[
(c_\infty - a_\nu)\alpha(\mu) = \sum_{\substack{f_\bullet \nu' = \nu \\ \nu'\neq \nu}} a_{\nu'}\alpha(\nu\wedge \nu')
\]
when $\mu$ is close enough to $\nu$.
Note that the right hand side of this expression is constant as $\mu$ varies, whereas the left hand side will vary unless $c_\infty = a_\nu$.
Thus we must have $c_\infty = a_\nu$.
However, this implies the right hand side of the expression is $0$, which is impossible unless there are no $\nu'\neq \nu$ such that $f_\bullet \nu' = \nu$.
We conclude that $\nu$ is totally invariant. 

To prove the last statement in the proposition, recall that the tangent space at $\nu$ is in bijective correspondence with the points of $E$, and the tangent map $df_\bullet$ at $\nu$ is given by $f\colon E\to E$.
Direct computation shows that $a_\nu$ is precisely the degree of the map $f\colon E\to E$ (see also the proof of \cite[Lemma 1.10]{favre:holoselfmapssingratsurf}).
The proposition then follows from the fact that $a_\nu = c_\infty>1$.
\end{proof}

\begin{prop}
Shrinking $J$ if necessary, one can suppose $m = 2$, so that $J$ consists entirely of eigenvaluations for $f^2$.
\end{prop}
\begin{proof}
We may without loss of generality assume that $J$ contains a divisorial eigenvaluation $\nu_\star$ for $f$, since we have seen in \S3 that the corollary holds when $J$ contains an irrational eigenvaluation for $f$.
Let $V$ be the collection of tangent vectors at $\nu_\star$ in the direction of $J$; the set $V$ consists of either one or two tangent vectors.
Applying \hyperref[prop:tot_inv]{Proposition~\ref*{prop:tot_inv}} to $f^m$, we see that the vectors in $V$ are totally invariant for $df^m_\bullet$ at $\nu_\star$.
Let $\pi\colon X\to \C^2$ be a blowup over $0$ such that $\nu_\star = b_E^{-1}\ord_E$ for some exceptional prime $E$ of $\pi$.
The tangent map $df_\bullet$ at $\nu_\star$ is given by the holomorphic map $f\colon E\to E$.
By \hyperref[prop:tot_inv]{Proposition~\ref*{prop:tot_inv}}, the tangent map $f^m\colon E\to E$ has degree $>1$, and hence $f\colon E\to E$ must have degree $>1$.
To proceed, we will need to appeal to the following well-known result of one dimensional dynamics.

\begin{prop}\label{prop:onedim}
Let $k$ be an algebraically closed field of characteristic $0$ and suppose $R\colon \mathbb{P}^1_k\to \mathbb{P}^1_k$ is a rational map of degree $d\geq 2$.
Let $\ms{E}_R$ be the collection points $z\in \mathbb{P}^1_k$ such that $R^{-n}(z) = \{z\}$ for some $n\geq 1$.
Then $\ms{E}_R = \ms{E}_{R^n}$ for all $n\geq 1$.
Moreover, $\#\ms{E}_R \leq 2$ and every $z\in \ms{E}_R$ satisfies $f^{-2}(z) = \{z\}$.
\end{prop}

\begin{rem}\label{rem:charp}
\hyperref[prop:onedim]{Proposition~\ref*{prop:onedim}} fails when $k$ has characteristic $p>0$, though only slightly.
Indeed, if the proposition does not hold for $R\colon \mathbb{P}^1_k\to \mathbb{P}^1_k$, then necessarily $R$ is conjugate to an iterate of the Frobenius automorphism $z\mapsto z^p$, and $\ms{E}_R$ is countably infinite.
\end{rem}

Since every vector in $V$ is totally invariant for $df_\bullet^m$ and $\mathrm{char}(\C) = 0$, \hyperref[prop:onedim]{Proposition~\ref*{prop:onedim}} shows that every vector in $V$ is in fact totally invariant for $df_\bullet^2$.
Let $\vec{v}\in V$, and let $[\nu_\star,\nu]\subseteq J$ be a segment in the direction of $\vec{v}$.
Since $\vec{v}$ is $df^2_\bullet$-totally invariant, it is in particular $df^2_\bullet$-fixed.
Applying \hyperref[prop:regularity]{Proposition~\ref*{prop:regularity}(3)}, there is a subinterval $I = [\nu_\star, \mu]\subseteq [\nu_\star,\nu]$ such that $f^2_\bullet I\subseteq [\nu_\star,\nu]$ and integers $a,b,c,d\in \N$ for which $\alpha(f_\bullet^2\tau) = (a\alpha(\tau)+b)/(c\alpha(\tau) + d)$ for $\tau\in I$.
\hyperref[lem:mobius]{Lemma~\ref*{lem:mobius}} says precisely that either $f_\bullet^2$ acts as the identity on $I$, or else $\nu_\star$ is an attracting fixed point for $f^2_\bullet$.
The latter possibility case cannot happen, since every point of $I$ is an eigenvaluation for $f^m$.
Thus $I$ consists of eigenvaluations for $f^2$, and we may shrink $J$ to $I$. 
\end{proof}

\begin{prop}\label{prop:segment_basins}
Let $f\colon (\C^2,0)\to (\C^2,0)$ be a dominant superattracting holomorphic germ that does not fall into either situation 1 or 2 of \hyperref[thm:basins]{Theorem~\ref*{thm:basins}}.
Then $f$ falls into situation 3 of \hyperref[thm:basins]{Theorem~\ref*{thm:basins}}.
\end{prop}
\begin{proof}
Let $I\subset \mc{V}$ be any maximal (under inclusion) segment of fixed points of $f_\bullet^2$.
We have seen that such a segment exists.
We must show that for every $\nu\in \mc{V}_A$ there is a $\nu_\star\in I$ such that $f^{2n}_\bullet\nu\to \nu_\star$ in the thinness topology.
We will need the following lemma, similar to \hyperref[lem:directed_basins]{Lemma~\ref*{lem:directed_basins}}.

\begin{lem}\label{lem:directed_basins_2}
Let $\nu_\star\in I$ be divisorial, and let $V$ be the set of tangent vectors at $\nu_\star$ in the direction of $I$.
Then there is a function $\rho$ defined on tangent vectors $\vec{v}\notin V$ at $\nu_\star$ with values in $(0,+\infty]$ such that
\begin{enumerate}
\item[$1.$] if $\nu\in \mc{V}_A$ lies in the direction $\vec{v}\notin V$ from $\nu_\star$ and $d_A(\nu, \nu_\star)<\rho(\vec{v})$, then $f^{2n}_\bullet\nu\to \nu_\star$ in the thinness topology as $n\to \infty$, and
\item[$2.$] one has $\inf_{\vec{v}\notin V}\rho(\vec{v})>0$.
\end{enumerate}
\end{lem}
\begin{proof}
To ease notation, let $g = f^2$, so that $I$ is a set of fixed points of $g_\bullet$.
This proof is essentially the same as the proof of \hyperref[lem:directed_basins]{Lemma~\ref*{lem:directed_basins}}, except one must first prove the following: if $\vec{v}\notin V$ is a $dg_\bullet$-periodic tangent direction at $\nu_\star$, say with period $m$, then there is no segment $J = [\nu_\star, \nu]$ in the direction of $\vec{v}$ such that $g^{m}_\bullet$ acts as the identity on $J$.
Suppose for contradiction that such a $\vec{v}$ and $J$ did exist.
Up to shrinking $J$, we have seen that we may assume $m = 2$.
\hyperref[prop:tot_inv]{Proposition~\ref*{prop:tot_inv}} tells us that $\vec{v}$ is then $dg^2_\bullet$-totally invariant, as is the set $V$.
But $dg_\bullet$ has degree $>1$, so by \hyperref[prop:onedim]{Proposition~\ref*{prop:onedim}} there can be at most two $dg_\bullet^2$-totally invariant tangent vectors at $\nu_\star$.
It follows that $V$ consists of a single tangent vector, or in other words, that $\nu_\star$ is an end of $I$.
But then $I\cup J$ is a segment of eigenvaluations for $g$ strictly containing $I$, a contradiction of the maximality of $I$.
We conclude that no such $\vec{v}$ and $J$ can exist.
The proof now proceeds exactly as in \hyperref[lem:directed_basins]{Lemma~\ref*{lem:directed_basins}}.
\end{proof}

\begin{cor}\label{cor:segment_basin}
Let $\nu_\star$ be a divisorial valuation in $I$, and let $V$ be the set of tangent vectors at $\nu_\star$ in the direction of $I$.
Let $X = \mc{V}_A\smallsetminus \bigcup_{\vec{v}\in V} U(\vec{v})$.
Then $f^{2n}_\bullet\nu\to \nu_\star$ in the thinness topology for all $\nu\in X$.
\end{cor}
\begin{proof}
Let
\[
B = \{\nu\in X : f^{2n}_\bullet\nu\to \nu_\star \mbox{ in the thinness topology as }n\to \infty\}.
\]
Since $X$ is a closed set, \hyperref[prop:basin_closed]{Proposition~\ref*{prop:basin_closed}} implies that $B$ is closed.
Let $\rho$ be as in \hyperref[lem:directed_basins_2]{Lemma~\ref*{lem:directed_basins_2}} and let $r = \inf_{\vec{v}\notin V} \rho(\vec{v})$.
Then the open ball $B_A(\nu_\star, r)\cap X$ is contained in $B$, proving that $B$ is open as well.
Since $X$ is connected, we conclude $B = X$.
\end{proof}

We now continue with the proof of \hyperref[prop:segment_basins]{Proposition~\ref*{prop:segment_basins}}.
Let $\nu\in \mc{V}_A$.
If $\nu\in I$, then $f^{2n}_\bullet \nu\to \nu$ trivially as $n\to \infty$.
Assume then that $\nu\notin I$.
Let $\nu_\star$ be the closest point of $I$ to $\nu$, that is, $\nu_\star$ is the image of $\nu$ under retraction $r\colon \mc{V}\to I$.
Then $\nu_\star$ is divisorial, and $\nu$ does not lie in the same direction from $\nu_\star$ as $I$.
Thus \hyperref[cor:segment_basin]{Corollary~\ref*{cor:segment_basin}} gives that $f^{2n}_\bullet\nu\to \nu_\star$, completing the proof.
\end{proof}

\section{Proof of Theorem B}

In this section, we use the results from \S3 and \S4 to prove \hyperref[thm:weakstability]{Theorem~\ref*{thm:weakstability}}, the geometric counterpart to \hyperref[thm:recursionrelation]{Theorem~\ref*{thm:recursionrelation}}.
For the entirety of this section, we will assume that $f\colon (\C^2,0)\to(\C^2,0)$ is a \emph{finite} superattracting germ, or in other words that $f$ has no contracted curves.
We will consider modifications $\pi\colon X\to (\C^2,0)$ of the following simple form: first, let $\eta\colon Y\to (\C^2,0)$ be a composition of point blowups over the origin; then, let $\pi\colon X\to (\C^2,0)$ be the modification obtained by contracting some chains of exceptional prime divisors of $Y$ with self-intersection $\leq -2$.
Any $X$ obtained in this way has at worst Hirzebruch-Jung quotient singularities, and dominates the blowup of the origin in $\C^2$.
In particular all Weil divisors are $\Q$-Cartier (see, for example, \cite{artalbartolo-martinmorales-ortigasgalindo:cartierweilonquotientsingularities}).
In this section, the term \emph{modification} will refer specifically to a modification of this type.

Let $\pi\colon X\to (\C^2,0)$ be a modification, and let $\Div_\Q(\pi)$ denote the $\Q$-vector space generated by the irreducible components of $\pi^{-1}(0)$.
The germ $f$ induces a $\Z$-linear pullback map $f^*\colon \Div_\Q(\pi)\to \Div_\Q(\pi)$ in the following way.
First, let $\mu\colon X'\to X$ be any bimeromorphic model of $(\C^2,0)$ dominating $X$ such that $f$ lifts to a holomorphic map $\tilde{f}\colon X'\to X$.
Since the elements of $\Div_\Q(\pi)$ are $\Q$-Cartier, we obtain a $\Z$-linear pullback map $\tilde{f}^*\colon \Div_\Q(\pi)\to \Div_\Q(\pi\circ\mu)$.
The pullback $f^*\colon \Div_\Q(\pi)\to \Div_\Q(\pi)$ is defined as the composition $\mu_*\circ \tilde{f}^*$.
It does not depend on the choice of $\mu$, and hence is well-defined.
Similarly, one has pullbacks $(f^n)^*$ for all $n\geq 1$.

The germ $f$ induces a holomorphic map $f\colon \pi^{-1}(0)\to \pi^{-1}(0)$, so it makes sense to refer to the image $f(E)\subseteq \pi^{-1}(0)$ of an exceptional prime divisor $E\subseteq\pi^{-1}(0)$, even in the case when $E$ contains indeterminacy points for the meromorphic map $f\colon X\dashrightarrow X$.
The image $f(E)$ is either an exceptional prime divisor of $X$, or else a point in $\pi^{-1}(0)$.
Analogous to \emph{algebraic stability} for (global) complex dynamical systems, there is a characterization of when the pullbacks $(f^n)^*$ behave functorially in terms of the images $f^n(E)$ of exceptional prime divisors $E\subseteq \pi^{-1}(0)$.

\begin{lem}
Suppose $f,g\colon (\C^2,0)\to (\C^2,0)$ are two finite superattracting holomorphic germs, and let $\pi\colon X\to (\C^2,0)$ be a modification.
Suppose that no exceptional prime divisor of $\pi$ is contracted by $f$ to an indeterminacy point of $g\colon X\dashrightarrow X$.
Then $(g\circ f)^* = f^*\circ g^*$ on $\Div_\Q(\pi)$.
\end{lem}
\begin{proof}
Let $\mu\colon Y\to X$ be a modification of $X$ such that $g$ lifts to a holomorphic map $\tilde{g}\colon Y\to X$.
We may assume with no loss of generality that $\mu$ is an isomorphism over any non-indeterminacy point of $g\colon X\dashrightarrow X$.
Similarly, let $\eta\colon Z\to X$ be a modification such that $f$ lifts to a holomorphic map $\tilde{f}\colon Z\to Y$.
By definition, $(g\circ f)^* = \eta_*\tilde{f}^*\tilde{g}^*$ and $f^*\circ g^* = \eta_*\tilde{f}^*\mu^*\mu_*\tilde{g}^*$.
The key observation is that, if $D\in \Div_\Q(\pi\circ \mu)$, then $\mu^*\mu_*D$ has the same coefficients as $D$ along any exceptional prime of $Y$ that is not contracted to a point via $\mu$.
Suppose now that $F$ is an exceptional prime of $Z$ that is not contracted to a point via $\eta$.
Our assumption that $f(\eta(F))$ is not an indeterminacy point of $g$ and our assumption that $\mu$ is an isomorphism over non-indeterminacy points of $g$ imply that $\tilde{f}(F)$ is either (1) an exceptional prime of $Y$ that is not contracted to a point via $\mu$, or (2) a point in $Y$ contained only within exceptional primes of $Y$ which are not contracted to a point via $\mu$.
In particular, the coefficient of $\eta(F)$ in $(g\circ f)^*E$ for any exceptional prime $E$ of $X$ is completely determined by the coefficients of $\tilde{g}^*E$ along exceptional primes of $Y$ that are not contracted to a point via $\mu$; but the coefficients of $\tilde{g}^*E$ along these primes are the same as the coefficients of $\mu^*\mu_*\tilde{g}^*E$ along these primes.
Thus $\eta_*\tilde{f}^*\tilde{g}^* = \eta_*\tilde{f}^*\mu^*\mu_*\tilde{g}^*$.
\end{proof}

Using this lemma, \hyperref[thm:weakstability]{Theorem~\ref*{thm:weakstability}} is implied by the following theorem.

\begin{thm}\label{thm:weakstability2}
Let $f\colon (\C^2,0)\to (\C^2,0)$ be a finite superattracting holomorphic germ.
Then there is a modification $\pi\colon X\to (\C^2,0)$ dominating the blowup of the origin in $\C^2$ such that for every exceptional prime divisor $E\subseteq \pi^{-1}(0)$, one has that $f^n(E)$ is not an indeterminacy point of $f\colon X\dashrightarrow X$ for sufficiently large $n$.
If one replaces $f$ by $f^2$, then $X$ can be taken to be smooth.
\end{thm}

Before being able to prove this theorem, we need one additional notion, namely the \emph{center} of a valuation $\nu\in \mc{V}$ in $X$.
For any $\nu \in \mc{V}$, there is a (possibly non-closed) associated point $p\in \pi^{-1}(0)$, called the \emph{center} of $\nu$ in $X$.
Without going into details, the center can be described as follows.
If $\nu = \nu_E$ is a divisorial valuation associated to an exceptional prime divisor $E\subseteq \pi^{-1}(0)$, then the center of $\nu$ in $X$ is the generic point of $E$.
If $p\in \pi^{-1}(0)$ is a closed point contained in a unique exceptional prime $E$, then the $\nu\in \mc{V}$ with center $p$ are exactly those $\nu$ lying in the tangent direction at $\nu_E$ defined by $p$.
Finally, if $p\in \pi^{-1}(0)$ is a closed point contained in the intersection of some set of exceptional primes $E_i$, then the $\nu\in \mc{V}$ with center $p$ are exactly those $\nu$ that lie in the intersection of the tangent directions of the $\nu_{E_i}$ defined by $p$.
If $p\in \pi^{-1}(0)$ is a closed point, we will denote by $U(p)$ the (weak open) set of $\nu\in \mc{V}$ with center $p$.
The reason for introducing the center is so we can use the following valuative criterion for testing whether a point $p\in \pi^{-1}(0)$ is an indeterminacy point for $f\colon X\dashrightarrow X$.

\begin{lem}\label{lem:indet}
The meromorphic map $f\colon X\dashrightarrow X$ is holomorphic at a closed point $p\in \pi^{-1}(0)$ if and only if there is a closed point $q\in \pi^{-1}(0)$ such that $f_\bullet U(p)\subseteq U(q)$.
In this case, $f(p) = q$.
\end{lem}
\begin{proof}
This is proven in \cite[Proposition 3.2]{favre-jonsson:eigenval} when $\pi\colon X\to (\C^2,0)$ is a composition of point blowups.
However, the proof is equally valid for the slightly more general modifications $\pi\colon X\to (\C^2,0)$ we are considering here, as the surfaces $X$ are normal.
\end{proof}

We are now ready to prove \hyperref[thm:weakstability2]{Theorem~\ref*{thm:weakstability2}}, and hence \hyperref[thm:weakstability]{Theorem~\ref*{thm:weakstability}}.

\begin{proof}[Proof of {\hyperref[thm:weakstability2]{Theorem~\ref*{thm:weakstability2}}}]
We will split the proof of the theorem into four cases, depending on the nature of the eigenvaluations of $f$.

\smallskip
\textbf{Case 1:}
First assume that $f$ has a unique, globally attracting eigenvaluation $\nu_\star$ that is either an end or an irrational valuation.
By \cite[Theorem 5.1 and Proposition 5.2]{favre-jonsson:eigenval}, there exists some blowup $\pi\colon X\to (\C^2,0)$ over $0$ and a closed point $p\in \pi^{-1}(0)$ such that $f$ is holomorphic at $p$, $f(p) = p$, and $U(p)$ is a neighborhood of $\nu_\star$.
If $E\subseteq \pi^{-1}(0)$ is an exceptional prime, then by \hyperref[thm:basins]{Theorem~\ref*{thm:basins}} the divisorial valuation $\nu_E$ lies in the basin of attraction of $\nu_\star$, and hence in particular $f^n_\bullet\nu_E\in U(p)$ for $n$ large enough.
If $\xi$ denotes the generic point of $E$, it follows that
\[
f^n(\xi) = f^n(\Center_X(\nu_E)) = \Center_X(f^n_\bullet \nu_E) = p
\]
for $n$ large enough.
Since $p$ is not an indeterminacy point of $f$, the theorem holds for the blowup $\pi\colon X\to (\C^2,0)$.
This completes the proof in case 1.

In each of the next three cases, we will use the following lemma.

\begin{lem}\label{lem:good_blowup}
Suppose that $\eta\colon Y\to (\C^2,0)$ is a blowup over the origin containing two exceptional primes $E$ and $F$ such that $f(E) = F$ and $f(F) = E$, where here we allow the possibility that $E = F$.
Then there is a blowup $\pi\colon X\to (\C^2,0)$ over the origin dominating $\eta$ such that $f$ is holomorphic at every periodic point $p\in E\cup F$ except possibly if $U(p)$ contains a segment of eigenvaluations for $f^2$.
\end{lem}
\begin{proof}
The following proof is identical in spirit to the proof of \cite[Lemma 4.6]{favre-jonsson:eigenval}.
If $E\neq F$, then we may assume with no loss of generality that $E$ and $F$ do not intersect in $Y$, as otherwise we may replace $Y$ with the blowup of $Y$ at the intersection point $E\cap F$.
Suppose $p\in Y$ is a periodic point lying in $E\cup F$ such that the orbit $p$ contains an indeterminacy point of $f\colon Y\dashrightarrow Y$.
Assume, moreover, that $U(p)$ contains no segment of eigenvaluations for $f^2$.
To prove the lemma, it suffices to show that there is a blowup $\mu\colon X\to Y$ over the orbit of $p$ such that $f\colon X\dashrightarrow X$ is holomorphic along the orbit of $p$, because this introduces no other indeterminacy points in $E\cup F$, and hence strictly reduces the number of periodic indeterminacy points in $E\cup F$.

Let $n$ be the minimal period of $p$, and set $p_i = f^i(p)$ for each $i = 0,\ldots, n-1$.
For each $i$ such that $p_i\in E$, blowup the point $p_i$ to obtain an exceptional prime $F_i$, and then successively blowup the strict transform of $E$ with the previously obtained exceptional divisor some large number $m_i$ of times, to be determined later.
Let $E_i$ denote the last exceptional prime obtained in the process.
Do the same procedure for the points $p_i$ lying in $F$, with $E$ replaced by $F$.
Let $\pi\colon X\to (\C^2,0)$ denote the composition of all of these blowups.
If $i$ is such that $p_i\in E$, then the segment $J_i:=[\nu_E, \nu_{E_i}]\subset\mc{V}$ lies in the tangent direction at $\nu_{E}$ defined by $p_i$, and has length $L_i = b_E^{-1}(m_ib_E + b_{F_i})^{-1}$ in the skewness metric; the analogous statement holds for those $i$ with $p_i\in F$, just with $E$ replaced by $F$.
If the $m_i$ are chosen large enough, the lengths $L_i$ can be taken as small as desired, and hence we may assume that $f_\bullet\colon J_i\to f_\bullet(J_i)$ acts as a M\"{o}bius transformation with non-negative integer coefficients in the skewness parameterization.
Let $s_i$ be the derivative of this M\"{o}bius transformation at $\alpha(\nu_E)$ or $\alpha(\nu_F)$.
While we have little control over the $s_i$ individually, the product $s_0\cdots s_{n-1}$ must be $<1$.
This is because $f^n_\bullet$ is either a contraction on $J_0$ or else the identity by \hyperref[lem:mobius]{Lemma~\ref*{lem:mobius}}; the latter case is ruled out by our assumption that $U(p)$ contains no segment of eigenvaluations for $f^2$ and the results of \S4. 

Let $\eps>0$ be small enough that $s_0\cdots s_{n-1}<(1-2\eps)^n$.
We now prove that the $m_i$ can be chosen so that the lengths $L_i$ are arbitrarily small satisfying $s_iL_i<(1-\eps)L_{i+1}$, the indices taken modulo $n$.
Since, if the $J_i$ are small, $f_\bullet \colon J_i\to f_\bullet(J_i)$ is approximately a linear expansion by $s_i$ in the skewness metric, this condition will imply that $f_\bullet(J_i)\subset J_{i+1}$.

Take $m_0$ to be any large integer, and then choose $m_1$ such that $(1-2\eps)L_{1}<s_0L_0<(1-\eps)L_1$.
Such an $m_1$ will necessarily also be large.
Now choose $m_2$ such that $(1-2\eps)L_2 < s_1L_1 < (1-\eps)L_2$.
Again, if $m_0$ was chosen large enough, then $m_2$ will necessarily be large.
Continue in this manner until one has large integers $m_0,\ldots, m_{n-1}$ such that $(1-2\eps)L_i<s_{i-1}L_{i-1}<(1-\eps)L_i$ for each $i = 1,\ldots, n-1$.
Note that this implies
\[
s_{n-1}L_{n-1} < \frac{s_0\cdots s_{n-1}}{(1-2\eps)^{n-1}} L_0 < (1-2\eps)L_0 < (1-\eps)L_0,
\]
so that in fact $s_iL_i< (1-\eps)L_{i+1}$ for all $i$ (taken modulo $n$).
If we take the $L_i$ small enough, then $f_\bullet$ expands $\alpha$-distances roughly by a factor of $s_i$ on $J_i$, so these inequalities prove that in fact $f_\bullet(J_i)\subseteq J_{i+1}$.
If the $L_i$ are small enough, then $f_\bullet$ is order preserving on those valuations $\nu\in \mc{V}$ whose retraction onto $J_i$ is contained in the interior of $J_i$.
By \hyperref[lem:indet]{Lemma~\ref*{lem:indet}}, this says precisely that $f\colon X\dashrightarrow X$ is holomorphic at each of the points $p_i$.
\end{proof}

We now continue with the proof of \hyperref[thm:weakstability2]{Theorem~\ref*{thm:weakstability2}}.
The next case we consider is the following.

\smallskip
\textbf{Case 2:}
Assume that $f$ has a unique, globally attracting divisorial eigenvaluation $\nu_\star$.
Let $\pi\colon X\to (\C^2,0)$ be a blowup over the origin containing the exceptional prime $E_\star$ associated to $\nu_\star$.
Since $\nu_\star$ is an eigenvaluation, $f(E_\star) = E_\star$.
According to \hyperref[lem:good_blowup]{Lemma~\ref*{lem:good_blowup}}, up to taking a further blowup we may assume that no indeterminacy points of $f\colon X\dashrightarrow X$ in $E_\star$ are periodic.
Let $U\subseteq\mc{V}$ be the open set consisting of all $\nu$ whose center in $X$ is contained in $E_\star$.
This is a weak open neighborhood of $\nu_\star$, and hence by \hyperref[thm:basins]{Theorem~\ref*{thm:basins}}, the orbit of every quasimonomial valuation $\nu$ eventually lies in $U$.
In particular, if $\nu = \nu_E$ is the divisorial valuation associated to an exceptional prime $E\subseteq \pi^{-1}(0)$ and $\xi$ is the generic point of $E$, then
\[
f^n(\xi) = f^n(\Center_X(\nu_E)) = \Center_X(f^n_\bullet\nu_E) \in E_\star
\]
for large enough $n$. Either $f^m(\xi)$ will be the generic point of $E_\star$ for some $m$, in which case $f^n(E) = E_\star$ for all $n\geq m$, or else $f^n(E)$ is a closed point of $E_\star$ for all large enough $n$.
Since indeterminacy points of $f$ in $E_\star$ are not periodic, we conclude that $f^n(E)$ is not an indeterminacy point of $f$ for large enough $n$, completing the proof in this case.

\smallskip
\textbf{Case 3:}
We now assume that $f$ has a nontrivial segment of eigenvaluations.
Let $I$ be the maximal segment of fixed points for $f_\bullet$.
Let $\pi\colon X\to (\C^2,0)$ be any blowup containing the exceptional primes associated to the minimal element of $I$ and, if $I$ has divisorial endpoints, the endpoints of $I$.
Let $E_1,\ldots, E_r$ denote the exceptional primes in $X$ such that $\nu_{E_i}\in I$.
Each of these is fixed by $f$, and hence using \hyperref[lem:good_blowup]{Lemma~\ref*{lem:good_blowup}} we may pass to a further blowup if necessary to ensure that $f$ is holomorphic at all periodic points of the $E_i$, except possibly those points lying in the intersection of two $E_i$.
However, by construction if $p$ lies in the intersection of two $E_i$, then $f_\bullet U(p)\subseteq U(p)$, so $f$ is holomorphic at these points as well.
If $E$ is any exceptional prime of $\pi$, then \hyperref[thm:basins]{Theorem~\ref*{thm:basins}} gives us that the center of $f^n_\bullet \nu_E$ is eventually contained in $E_i$ for some $i$.
Arguing as in case 2, the theorem follows.

\smallskip
\textbf{Case 4:}
Finally, assume that $f$ has a nontrivial segment of eigenvaluations for $f^2$, but not for $f$.
Let $I$ be the maximal interval of fixed points for $f^2_\bullet$.
Let $\eta\colon Y\to (\C^2,0)$ be any blowup containing the exceptional primes associated to the minimal element of $I$ and, if $I$ has divisorial endpoints, the endpoints of $I$.
Let $\eta'\colon Y'\to (\C^2,0)$ be the minimal further blowup such that if $E$ is the strict transform of an exceptional prime of $Y$ with $\nu_E\in I$, then there is an exceptional prime $E'$ of $Y'$ such that $f(E) = E'$.
Since $\eta'$ is chosen minimal with this property, any exceptional prime $F$ of $Y'$ that does not appear in $Y$ is such that $\nu_F\in I$.
Let $S$ denote the set of strict transforms of exceptional primes $E$ of $Y$ with $\nu_E \in I$ and their images $f(E) = E'$.
Since $f^2(E) = E$ for any $E\in S$, \hyperref[lem:good_blowup]{Lemma~\ref*{lem:good_blowup}} allows us to find a further blowup $\pi'\colon X'\to (\C^2,0)$ such that $f$ 
is holomorphic at all periodic points of any exceptional prime $E\in S$, except possibly at those points corresponding to tangent directions at $\nu_E$ into $I$.
Label the primes in $S$ as $E_1,\ldots, E_r$ so that $[\nu_{E_1}, \nu_{E_2}]$, $[\nu_{E_2},\nu_{E_3}],\ldots, [\nu_{E_{r-1}}, \nu_{E_r}]$ are abutting subintervals of $I$.
We then let $\pi\colon X\to (\C^2,0)$ be the modification obtained by contracting the chains of exceptional primes of $X'$ corresponding to divisorial valuations contained in the open intervals $(\nu_{E_i}, \nu_{E_{i+1}})$.
The divisors $E_1,\ldots, E_r$ now form a chain in $X$, and \hyperref[lem:indet]{Lemma~\ref*{lem:indet}} implies that $f$ is holomorphic at each of the points of intersection $E_i\cap E_{i+1}$.
It follows that $f$ is holomorphic at all the periodic points of the $E_i$.
If $E$ is any exceptional prime of $X$, then \hyperref[thm:basins]{Theorem~\ref*{thm:basins}} tells us that $f^n_\bullet(\nu_E)$ will eventually have center lying within $E_i\cup f(E_i)$ for some 
index $i$.
Arguing as before, it follows that $f^n(E)$ is not an indeterminacy point of $f$ for $n$ sufficiently large.
This completes the proof.
\end{proof}

\begin{cor}\label{cor:recursion}
Let $f\colon (\C^2,0)\to (\C^2,0)$ be a finite superattracting germ.
Then the attraction rates $c(f^n)$ eventually satisfy a linear integral recursion formula.
\end{cor}
\begin{proof}
Let $\pi_0\colon X_0\to (\C^2,0)$ be the blowup of the origin in $\C^2$, and let $E$ denote its unique exceptional prime.
It is easy to see that the attraction rates $c(f^n)$ are given by the intersection numbers $c(f^n) = -E\cdot (f^n)^*E$, where here $(f^n)^*$ is the pullback $\Div_\Q(\pi_0)\to \Div_\Q(\pi_0)$.
Let $\pi\colon X\to (\C^2,0)$ be the modification given in \hyperref[thm:weakstability]{Theorem~\ref*{thm:weakstability}}.
Since $\pi$ dominates $\pi_0$, there is a holomorphic map $\mu\colon X\to X_0$ such that $\pi = \pi_0\circ \mu$.
Let $D = \mu^*E$.
It follows that $c(f^n) = -D\cdot (f^n)^*D$, where now $(f^n)^*$ denotes the pullback $\Div_\Q(\pi)\to \Div_\Q(\pi)$.
\hyperref[thm:weakstability]{Theorem~\ref*{thm:weakstability}} then tells us that there is an integer $N\geq 0$ such that $c(f^n) = -D\cdot (f^*)^{n-N}(f^N)^*D$ for $n>N$.
Since $f^*$ is $\Z$-linear, an easy linear algebra argument allows us to conclude that $c(f^n)$ satisfies an integral linear recursion formula when 
$n>N$.
\end{proof}

One consequence of \hyperref[cor:recursion]{Corollary~\ref*{cor:recursion}} is that, if the modification $\pi\colon X\to (\C^2,0)$ from \hyperref[thm:weakstability]{Theorem~\ref*{thm:weakstability}} were such that $(f^n)^* = (f^*)^n$ for all $n\geq 1$, then the attraction rates $c(f^n)$ would satisfy an integral linear recursion formula for all $n\geq 1$.
We will use this observation to show that one cannot always find such an $X$ in \hyperref[ex:not_stable]{Example~\ref*{ex:not_stable}}.

\section{Proof of Theorem A}

In this section, we use the results from \S3 and \S4 to prove \hyperref[thm:recursionrelation]{Theorem~\ref*{thm:recursionrelation}}.
Indeed, we shall prove a more general and precise result, from which \hyperref[thm:recursionrelation]{Theorem~\ref*{thm:recursionrelation}} follows by taking $\nu = \ord_0$.

\begin{thm}\label{thm:recursioncases} Let $f\colon (\C^2,0)\to (\C^2,0)$ be a dominant superattracting fixed point germ.
Suppose $\nu\in \mc{V}$ is a valuation of finite thinness.
Let $c_n := c(f^n,\nu)$ for all $n\geq 1$.
\begin{enumerate}
\item[$1.$] If $f$ has an end eigenvaluation, then the sequence $c_n$ eventually satisfies an integral linear recursion relation of order $1$.
\item[$2.$] If $f$ has a unique, globally attracting irrational eigenvaluation, then the sequence $c_n$ eventually satisfies an integral linear recursion relation of order $2$.
\item[$3.$] If $f$ has a unique, globally attracting divisorial eigenvaluation, then for some integer $m\geq 1$ independent of $\nu$, the sequence $c_{nm}$ eventually satisfies an integral linear recursion relation of order~$2$.
\item[$4.$] If there is a non-degenerate segment of eigenvaluations for $f^2$, then for some integer $m\geq 1$ independent of $\nu$, the sequence $c_{nm}$ eventually satisfies an integral linear recursion relation of order $2$.
\end{enumerate}
It follows that in every case, the sequence $c_n$ eventually satisfies an integral linear recursion relation, though possibly of order $>2$.
\end{thm}

This theorem is the local analogue of results proved by Favre-Jonsson in the setting of polynomial maps of $\C^2$ \cite[Corollaries 3.3, 4.2, 4.4]{favre-jonsson:dynamicalcompactifications}.
In their proofs of these results, they use normal forms of rigid holomorphic germs in dimension $2$.
We could, in fact, argue the same way here, but we will give a different proof, using the geometric methods discussed in \S1.

\begin{lem}\label{lem:geometric}
Let $f\colon (\C^2,0)\to (\C^2,0)$ be a dominant superattracting holomorphic fixed point germ.
Let $\pi\colon X\to \C^2$ be a blowup over $0$ and suppose that $E$ and $F$ are exceptional primes of $\pi$ that intersect transversely.
Let $\nu_E, \nu_F\in \mc{V}$ be the divisorial valuations corresponding to $E$ and $F$.
Suppose that the interval $I = [\nu_E, \nu_F]$ is invariant in the sense that $f_\bullet I\subseteq I$.
Then for every $\nu\in I$ the sequence $c_n:= c(f^n, \nu)$ satisfies an integral linear recursion relation of order at most $2$.
\end{lem}
\begin{proof}
Let $p = E\cap F$, and choose local coordinates $(z,w)$ of $X$ at $p$ such that $E = \{z = 0\}$ and $F = \{w = 0\}$.
Then any $\nu\in I$ is a monomial valuation at $p$ with some weights $r,s>0$ with respect to the coordinates $z$ and $w$.
Since $I$ is invariant, it follows that $f^n_*\nu$ is also a monomial valuation at $p$ with some weights $r_n, s_n$ with respect to the coordinates $z$ and $w$ for all $n\geq 1$.
In particular, one has that $Z_{f^n_*\nu, \pi} = r_n\check{E} + s_n\check{F}$ where we recall that $\check{E}\in \Div(\pi)$ denotes the unique exceptional divisor such that $(\check{E}\cdot E) = 1$ and $(\check{E}\cdot G) = 0$ for all exceptional primes $G\neq E$ of $\pi$ (with $\check{F}$ defined similarly).
Thus $r_n = (Z_{f^n_*\nu, \pi}\cdot E)$ and $s_n = (Z_{f^n_*\nu,\pi}\cdot F)$. 

We will now compute $r_n$ and $s_n$ in terms of $r_{n-1}$ and $s_{n-1}$.
Let $\pi'\colon X'\to \C^2$ be a blowup over $0$ dominating $\pi$ such that $f\colon X'\to X$ is holomorphic.
Let $\mu\colon X'\to X$ be the holomorphic map with $\pi' = \pi\circ \mu$.
By \cite[Lemma 1.10]{favre:holoselfmapssingratsurf}, $Z_{f^n_*\nu, \pi} = f_*Z_{f^{n-1}_*\nu, \pi'}$, so that
\[
r_n = (f_*Z_{f^{n-1}_*\nu,\pi'}\cdot E) = (f_*\mu^*(r_{n-1}\check{E} + s_{n-1}\check{F})\cdot E) = r_{n-1}(f_*\mu^*\check{E}\cdot E) + s_{n-1}(f_*\mu^*\check{F}\cdot E).
\]
Note that the quantities $(f_*\mu^*\check{E}\cdot E)$ and $(f_*\mu^*\check{F}\cdot E)$ are non-negative integers that do not depend on $n$.
One similarly obtains the formula
\[
s_{n} = r_{n-1}(f_*\mu^*\check{E}\cdot F) + s_{n-1}(f_*\mu^*\check{F}\cdot F).
\]
This proves there is a $2\times 2$ matrix $M$ of natural numbers such that $(r_n, s_n) = M(r_{n-1}, s_{n-1})$ for all $n\geq 1$.
To complete the proof, we note that
\[
c(f^n, \nu) = (f^n_*\nu)(\mf{m}) = -(Z_{f^n_*\nu, \pi}\cdot Z_{\ord_0,\pi}) = -r_n(\check{E}\cdot Z_{\ord_0,\pi}) - s_n(\check{F}\cdot Z_{\ord_0,\pi})  = b_Er_n + b_Fs_n.
\]
Thus the sequence $c_n = c(f^n, \nu)$ satisfies the integral linear recursion relation
\[
c_{n+2} = \mathrm{tr}(M)c_{n+1} - \det(M)c_n.
\]
This completes the proof.
\end{proof}

\begin{proof}[Proof of Theorem \ref{thm:recursioncases}]
We split the proof up into cases depending on the nature of the eigenvaluations of $f$.

\smallskip
\textbf{Case 1:}
First, let us assume that $f$ has an end eigenvaluation $\nu_\star$.
In this case $c_\infty$ is an integer, see \cite[Theorem 5.1]{favre-jonsson:eigenval}.
Moreover, there is a weak open neighborhood $U$ of $\nu_\star$ such that $c(f,-)$ is constant $\equiv c_\infty$ on $U$.
Since $f^n_\bullet\nu\to \nu_\star$ in the weak topology by \hyperref[thm:basins]{Theorem~\ref*{thm:basins}(1)}, one has $f^n_\bullet\nu\in U$ for large $n$.
Thus for large $n$ one has
\[
c_{n+1} = c(f^{n+1},\nu) = c(f, f^n_\bullet\nu)c(f^n,\nu) = c_\infty c_n,
\]
completing the proof in this case.

\smallskip
\textbf{Case 2:}
Now assume that $f$ has a unique, globally attracting irrational eigenvaluation $\nu_\star$.
As we saw in \S3.2, we can choose divisorial valuations $\nu_1<\nu_\star<\nu_2$ as close as desired to $\nu_\star$ such that the interval $I = [\nu_1,\nu_2]$ satisfies $f_\bullet I\Subset I$.
If $\vec{v}_i$ is the tangent vector at $\nu_i$ in the direction of $\nu_\star$, then by taking $I$ small enough we can assume $U := U(\vec{v}_1)\cap U(\vec{v}_2)$ is invariant and that $c(f^n,\nu) = c(f^n, r\nu)$ for all $n\geq 1$, where $r\colon U\to I$ is the retraction map.
Finally, we can choose the $\nu_i$ so that there exists a blowup $\pi\colon X\to \C^2$ over $0$ with exceptional primes $E_i$ intersecting transversely such that $\nu_i = b_{E_i}^{-1}\ord_{E_i}$, see \cite[Lemma 5.6]{favre-jonsson:eigenval}.
Thus \hyperref[lem:geometric]{Lemma~\ref*{lem:geometric}} applies to $I$, and we see that
\[
c_n := c(f^n,\nu) = c(f^n, r\nu)
\]
satisfies an integral linear recursion relation of order $2$ for all $\nu\in U$.
If $\nu\in \mc{V}_A$ is any valuation of finite thinness, then $f^n_\bullet\nu\in U$ for large enough $n$, and hence $c_n:= c(f^n,\nu)$ satisfies an integral linear recursion relation of order $2$ when $n$ is sufficiently large.

\smallskip
\textbf{Case 3:}
Now assume that $f$ has a unique, globally attracting divisorial eigenvaluation $\nu_\star$.
In this case,  $c_\infty = c(f,\nu_\star)$ is an integer by \cite[Proposition~2.5]{favre-jonsson:eigenval}.
Let $S_0$ be the collection of tangent vectors at $\nu_\star$ on which $c(f,-)$ is constant, and let $S_1$ be its (finite) complement.
Let $m$ be a common period for all periodic vectors in $S_1$; if there are no such periodic vectors, let $m =1$.
If $\nu\in \mc{V}_A$, then $f^{n}_\bullet\nu\to \nu_\star$ in the thinness topology.
There are three ways in which this convergence can happen:
\begin{enumerate}
\item[(a)] The sequence $\vec{v}_n$ of tangent vectors at $\nu_\star$ in the direction of $f^{n}_\bullet\nu$ lies in $S_1$ infinitely often.
\item[(b)] The sequence $\vec{v}_n$ lies in $S_0$ for all large enough $n$.
\item[(c)] $f^{N}_\bullet\nu = \nu_\star$ for some $N\geq 1$.
\end{enumerate}
In cases (b) and (c), one has $c(f,f^n_\bullet\nu) = c_\infty$ for large enough $n$, so that $c_{n+1} = c_\infty c_n$.
Suppose we are in case (a).
Then $f^{nm}_\bullet\nu\to \nu_\star$ along a fixed tangent direction $\vec{v}$.
From our work in \S3.3, we know there is a segment $I = [\nu_\star, \mu]$ in the direction of $\vec{v}$ such that $f^m_\bullet I\Subset I$ and $f^{mn}_\bullet I\to \nu_\star$.
If $\vec{w}$ is the tangent vector at $\mu$ in the direction of $\nu_\star$, then if $I$ is chosen small enough, $U = U(\vec{v})\cap U(\vec{w})$ is invariant, and $c(f,\tau) = c(f,r\tau)$ for all $\tau\in U$, where $r\colon U\to I$ is the retraction map.
Moreover, by choosing $\mu$ appropriately, there is a blowup $\pi\colon X\to \C^2$ over $0$ with exceptional primes $E$ and $F$ meeting transversely such that $\nu_\star = b_E^{-1}\ord_E$ and $\mu = b_F^{-1}\ord_F$.
Thus \hyperref[lem:geometric]{Lemma~\ref*{lem:geometric}} applies to $I$, and we see that $c(f^{nm}, \tau) = c(f^{nm},r\tau)$ satisfies an integral linear recursion relation of order $2$ for all $\tau\in U$.
For $n$ large enough, $f^{nm}_\bullet \nu\in U$, and hence $c_{nm} := c(f^{nm},\nu)$ satisfies an integral linear recursion relation of order at most $2$ for large $n$.

\smallskip
\textbf{Case 4:}
Finally, assume $f$ is such that there is a non-degenerate segment of eigenvaluations of $f^2$.
Let $I$ be the maximal (under inclusion) segment of fixed points for $f^2_\bullet$, let $r\colon \mc{V}\to I$ be the retraction, and let $\nu_\star$ be the minimal element of $I$.
One has $c(f^2, -)\equiv c_\infty^2$ along $I$.
Moreover $c_\infty^2$ is an integer, since $c(f^2,\nu)$ is an integer for any divisorial eigenvaluation of $f^2$ by \cite[Proposition~2.5]{favre-jonsson:eigenval}.
If $\nu\in I$ or $r(\nu)\neq \nu_\star$, then $c(f^2,\nu) = c_\infty^2$, so that $c_{2n} := c(f^{2n},\nu) = c_\infty^2c_{2(n-1)}$.
Now suppose that $\nu\in \mc{V}_A$ with $\nu\neq \nu_\star$ and $r(\nu) = \nu_\star$.
Let $S_0$ be the set of tangent vectors at $\nu_\star$ on which $c(f^2,-)$ is constant, and let $S_1$ be its (finite) complement.
Let $m\geq 1$ be a common period of all $f^2$-periodic vectors in $S_1$; if there are no such periodic vectors set $m = 1$.
We know that $f^{2n}_\bullet\nu\to \nu_\star$ in the thinness topology.
Since $\nu_\star$ is $f^2_\bullet$-totally invariant, there are only two possibilities for this convergence:
\begin{enumerate}
\item[(a)] The sequence $\vec{v}_n$ of tangent directions at $\nu_\star$ in the direction of $f^{2n}_\bullet\nu$ lies in $S_1$ infinitely often.
\item[(b)] The sequence $\vec{v}_n$ lies in $S_0$ for large enough $n$.
\end{enumerate}
Case (b) here is analogous to case (b) in the previous paragraph, and one obtains $c_{2n} = c_\infty^2 c_{2(n-1)}$ for all sufficiently large $n$.
Case (a) is analogous to case (a) in the previous paragraph, and one obtains that $c_{2mn}$ satisfies an integral linear recursion relation of order at most $2$ for large $n$.
\end{proof}

\begin{rem}
We should note that for holomorphic germs $f\colon (\C^2,0)\to (\C^2,0)$ that are not superattracting, \hyperref[thm:recursionrelation]{Theorem~\ref*{thm:recursionrelation}} is trivial, since $c(f^n,\nu)=1$ for any $n$ and any $\nu\in \mc{V}$.
In particular, $f_\bullet = f_*$ in this case.

There is a dichotomy in the dynamics of $f_*$ in the non-superattracting case, depending on the nature of the differential $df_0$: either $df_0$ is invertible, or $df_0$ has exactly one non-zero eigenvalue.
In the latter case, one can show (see \cite[Theorem 0.7]{ruggiero:rigidification}) that there exists a unique curve eigenvaluation $\nu_\star$, and every other valuation $\nu \in \mc{V}$ is (weakly) attracted to $\nu_\star$, except for at most one curve eigevaluation $\nu_D$, associated to the ``stable manifold'' of the eigenvalue $0$ of $df_0$.
This valuation is fixed (and repelling) if $f(D)=D$, or attracted to $\nu_\star$ if $f(D)=0$, i.e., if it is a contracted curve valuation.

If $f$ is invertible, $f_*\colon \mc{V}\to \mc{V}$ is a bijection, and skewness and thinness are preserved by the action of $f_*$.
In the case of the skewness, this can be seen directly from \hyperref[def:skewthin]{Definition~\ref*{def:skewthin}} and the fact that $f^*\mf{m} = \mf{m}$.
For thinness, it can be seen from the \hyperref[jacobian_formula]{Jacobian formula}.
The set $S$ of all $f_*$-periodic $\nu\in \mc{V}$ forms a (not necessarily finite) subtree of $\mc{V}$ containing $\ord_0$.
For all $\nu\in \mc{V}$, there is an $m\geq 1$ and a $\nu_\star\in S$ of period $m$ such that $f^{mn}_*\nu\to \nu_\star$ weakly as $n\to \infty$.
Because skewness and thinness are preserved by $f_*$, this cannot be strengthened to convergence in either the thinness or skewness topologies.
\end{rem}

\section{Examples}

In this section we provide some worked examples of deriving the recursion formula for $c_n := c(f^n)$ from the dynamics of $f$ on the valuative tree $\mc{V}$.
In Examples \ref{ex:mono1}, \ref{ex:mono2}, \ref{ex:mono3}, \ref{ex:mono4} and \ref{ex:mono5} we consider \emph{monomial maps} $f$, that is, germs of the form $f(x,y) = (x^ay^b,x^cy^d)$ for some $a,b,c,d\in \N$ with $ad - bc\neq 0$.
These examples are simple enough that the sequence $c_n$ can be easily studied without appealing to dynamics on $\mc{V}$.
Nonetheless, for illustrative purposes we will study their dynamics.
In the remaining examples we consider some non-monomial germs; the dynamics in Examples \ref{ex:kato} and \ref{ex:segment} are fairly straightforward to analyze, while those in Examples \ref{ex:CAR}, \ref{ex:high_order}, and \ref{ex:high_order2} are more complicated.
The germ considered in \hyperref[ex:CAR]{Example~\ref*{ex:CAR}} is the same as that in \cite[Example 4.1]{casasalvero-roe:iteratedinverseimages}; we include it so that their methods can be compared to those of the present article.
In Examples \ref{ex:high_order} and \ref{ex:high_order2} we show that minimal order $m$ of a linear recurrence eventually satisfied by the $c_n$ can be arbitrarily large; indeed, it can be any positive integer.
Finally, in \hyperref[ex:not_stable]{Example~\ref*{ex:not_stable}} we give an example of a germ with no local algebraically stable model.

Monomial maps $f(x,y) = (x^ay^b,x^cy^d)$ necessarily preserve \emph{monomial valuations} in the coordinates $(x,y)$.
A monomial valuation in the coordinates $(x,y)$ is a valuation $\nu_{s,t}\in\hat{\mc{V}}^*$ of the form
\[
\nu_{s,t}\left(\sum \lambda_{\alpha\beta} x^\alpha y^\beta\right) := \min\{\alpha s + \beta t : \lambda_{\alpha\beta}\neq 0\},
\]
where $s,t>0$.
One easily checks that $f_*\nu_{s,t} = \nu_{as + bt, cs+dt}$.
The normalized monomial valuations $\nu_{s,t}\in \mc{V}$ are those for which $\min\{s,t\} = 1$; they make up the segment $(\nu_x, \nu_y)$ in $\mc{V}$.
Thus $f_\bullet$ maps this segment into itself.

\begin{ex}\label{ex:mono1}
Let $f(x,y) = (x^d,y^d)$, where $d\geq 2$.
For this monomial map, every point of the segment $[\nu_x, \nu_y]$ is fixed for $f_\bullet$.
In particular, $\ord_0$ is fixed, so that
\[
c_{n} = c(f^n,\ord_0) = c(f,\ord_0)^n = d^n.
\]
Thus the sequence $c_n$ satisfies the order 1 recursion relation $c_{n+1} = dc_n$.
\end{ex}

\begin{ex}\label{ex:mono2}
Let $f(x,y) = (y^b,x^c)$ for $b,c\geq 2$.
Then $f^2(x,y) = (x^{bc}, y^{bc})$, and hence every point of the segment $[\nu_x,\nu_y]$ is fixed by $f_\bullet^2$.
It follows that
\[
c_{n+2} = c_2 \cdot c(f^n, f^2_\bullet\ord_0) = bc\cdot c(f^n, \ord_0) = bc\cdot c_n,
\]
an order 2 recursion relation.
Notice that when $b=c$, then $\ord_0$ is fixed by $f_\bullet$, and the sequence $c_n$ satisfies the order 1 recursion relation $c_{n+1}=b c_n$.
\end{ex}

\begin{ex}\label{ex:mono3}
Let $f(x,y) = (x^a,y^d)$ where $a>d\geq 2$.
In this case $\nu_x$ is the unique curve eigenvaluation for $f$.
Moreover, the segment $I = [\ord_0,\nu_x]$ is invariant in the sense that $f_\bullet I\subseteq I$.
Since $c(f,-)\equiv d$ on $I$, it follows that
\[
c_{n+1} = c_n\cdot c(f, f^n_\bullet\ord_0) = dc_n,
\]
an order $1$ recursion relation.
\end{ex}

\begin{ex}\label{ex:mono4}
Let $f(x,y) = (x^2y,x^2y^3)$.
The divisorial valuation $\nu_{1,2}$ is the unique eigenvaluation for $f$.
Moreover, the segment $I = [\ord_0,\nu_{1,2}]$ satisfies $f_\bullet I\subset I$.
We know by \hyperref[lem:geometric]{Lemma~\ref*{lem:geometric}} that this implies that $c_n$ satisfies a recursion relation of order 2.
This relation can be derived as follows.
In the skewness parameterization, $I$ is identified with the interval $[1,2]\subset\R$, and $f_\bullet \colon I\to I$ is identified with the map $\varphi\colon [1,2]\to [1,2]$ given by
\[
\varphi(t) = \frac{2+3t}{2+t}.
\]
Similarly, $c(f,-)\colon I\to \R$ is given by $t\in [1,2]\mapsto 2+t$.
If we set $t_n:= \alpha(f^n_\bullet \ord_0)$, one then has
\begin{align*}
c_{n+2} & = c_n\cdot c(f,f^{n}_\bullet\ord_0)\cdot c(f,f^{n+1}_\bullet\ord_0) = c_n(2+t_n)(2+t_{n+1}) = c_n(2+t_n)\left(2 + \frac{2+3t_n}{2+t_n}\right)\\
& = (6 + 5t_n)c_n = (5(2+ t_n) - 4)c_n = 5c_{n+1} - 4c_n.
\end{align*}
This recursion relation is easily verified by hand.
\end{ex}

\begin{ex}\label{ex:mono5}
Let $f(x,y) = (xy, x^2y)$.
The irrational valuation $\nu_{1,\sqrt{2}}$ is the unique eigenvaluation for $f$.
Similarly to \hyperref[ex:mono4]{Example~\ref*{ex:mono4}}, the interval $I = [\ord_0,\nu_{1,2}]$ is invariant for $f_\bullet$.
With respect to the skewness parameterization one has $I\cong [1,2]\subset\R$, and $f_\bullet\colon I\to I$ is given by $\varphi\colon [1,2]\to [1,2]$, 
\[
\varphi(t) = \frac{2+t}{1+t}.
\]
Furthermore, $c(f,-)\colon I\to \R$ is given by $t\in [1,2]\mapsto 1+t$.
By a similar derivation as was carried out in \hyperref[ex:mono4]{Example~\ref*{ex:mono4}}, one can see $c_{n+2} = 2c_{n+1}  + c_n$.
\end{ex}

In the remaining examples, we will use the following notation. If $\phi\in \mf{m}$ is irreducible and $t\geq 1$, we let $\nu_{\phi,t}$ denote the unique valuation in the segment $[\ord_0, \nu_\phi]$ with skewness $t$.

\begin{ex}\label{ex:kato}
We now consider a germ $f$ with topological degree $1$:
\[
f(x,y)=\left(\frac{x^2}{1+xy},\frac{x}{1+xy}\right).
\]
Such a germ is called a \emph{strict germ}.
These germs are used in the construction of Kato surfaces, and can be decomposed as $f = \pi\circ \sigma$, where $\pi\colon X\to (\C^2,0)$ is a blow-up over the origin, and $\sigma\colon (\C^2,0)\to X$ is a local biholomorphism onto a neighborhood of a smooth point $\sigma(0)\in X$, see \cite{dloussky:phdthesis}.
Superattracting strict germs are classified in \cite{favre:rigidgerms}.
For this particular germ $f$, one can take $\pi\colon X\to (\C^2,0)$ to be a composition of three point blowups. 

It can be shown that $f$ has a unique, infinitely singular eigenvaluation.
Moreover, if $\vec{v}$ is the tangent vector at $\nu_{x,2}$ in the direction of $\nu_{x - y^2}$, then $f_\bullet \mc{V}\Subset U(\vec{v})$.
Since $c(f,-)\equiv 2$ on $U(\vec{v})$, it follows that $c_{n+1} = c_n\cdot c(f, f^{n}_\bullet \ord_0) = 2c_n$ for all $n\geq 1$. 
\end{ex}

\begin{ex}\label{ex:segment}
Let $f(x,y) = (y+x^2,y^2)$.
For this germ, the entire segment $[\nu_{y,2}, \nu_y]$ is pointwise fixed by $f_\bullet$, and the function $c(f,-)$ is locally constant away from the segment $[\ord_0, \nu_{y+x^2,4}]$.
By \hyperref[prop:tot_inv]{Proposition~\ref*{prop:tot_inv}}, every point of $[\nu_{y,2},\nu_y]$ is totally invariant for $f_\bullet$.
If $\vec{v}_n$ is the tangent vector at $\nu_{y,2}$ in the direction of $f^n_\bullet\ord_0$ for each $n\geq 0$, it follows that $\vec{v}_n = df_\bullet^n\vec{v}_0$.
We will show that $c(f,-)\equiv 2$ on $U(\vec{v}_n)$ for all $n\geq 1$, and thus that $c(f, f^n_\bullet\ord_0) = 2$ for all $n\geq 1$.
To do this, we will explicitly compute the tangent map $df_\bullet$.

One can identify the tangent space at $\nu_{y,2}$ with $\pr^1_\C = \C\cup\{\infty\}$ as follows.
For each tangent vector $\vec{v}\neq \vec{v}_0$, there is a unique $\theta\in \C$ such that $\nu_{y - \theta x^2}$ lies in the direction $\vec{v}$.
Identifying this $\vec{v}$ with $\theta$ and $\vec{v}_0$ with $\infty$, the tangent map $df_\bullet\colon \pr^1\to \pr^1$ is a rational map.
One can easily verify by hand that $f(\{y  = \theta x^2\}) = \{y = R(\theta)x^2\},$ where $R(\theta) = \theta^2/(1+\theta)^2$.
Thus the tangent map $df_\bullet$ is given map $\theta\mapsto R(\theta)$.
There are only two tangent directions at $\nu_{y,2}$ in which $c(f,-)$ is not constant, namely those corresponding to $\theta = \infty$ and $\theta = -1$.
It is easy to check that $R^n(\infty)\neq \infty, -1$ for each $n\geq 1$, and thus $c(f,-)$ is constant in the direction $\vec{v}_n$ identified with $R^n(\infty)$ for each $n\geq 1$, proving that $c(f, f^n_\bullet\ord_0) = c(f,\nu_{y,2}) = 2$ for $n\geq 1$.
It follows that $c_{n+1} = c_n\cdot c(f, f^n_\bullet \ord_0) = 2c_n$ for all $n\geq 1$.
\end{ex}

\begin{ex}[{\cite[Example 4.1]{casasalvero-roe:iteratedinverseimages}}]\label{ex:CAR}
Let $f(x,y) = (y^2, y^4 - x^5)$.
For this germ, there is a finite subtree $T$ of $\mc{V}$ that is invariant for $f_\bullet$, namely the tree with endpoints $\ord_0$, $\nu_{y,5/2}$, $\nu_{y-x^2, 5/2}$, and $\nu_{y^4-x^5, 41/32}$, see \hyperref[fig]{Figure~\ref*{fig}} for an illustration of $T$, as well as a convenient labeling of edges of $T$.
The action of $f_\bullet$ on $T$ takes edge I to edge II, edge II to edge III, edge III to edge IV, and edge IV to edge I.
Edge V is mapped into itself by $f_\bullet$, and is pointwise fixed by $f^2_\bullet$.
In terms of the skewness parameterization, one can compute the action of $f_\bullet$ on $T$ to be
\[
\alpha(f_\bullet\nu) = \begin{cases}5/(2\alpha(\nu)) & \nu\mbox{ in edge I, IV, or V.}\\
(\alpha(\nu) + 18)/16 & \nu\mbox{ in edge II.}\\
8\alpha(\nu)/5 & \nu\mbox{ in edge III.}
\end{cases}
\]
Moreover, one has
\[
c(f,\nu) = \begin{cases} 2\alpha(\nu) & \nu\mbox{ in edge I, IV, or V.}\\ 4 & \nu\mbox{ in edge II.}\\
5/2 & \nu\mbox{ in edge III.}\end{cases}
\]
For the map $g = f^4$, each edge is invariant.
Moreover, using the above data it is easy to derive that $c(g,\nu) = 100$ for $\nu$ in edges II-V and that $c(g,\nu) = 10+72\alpha(\nu)$ for $\nu$ in edge I.
In the skewness parameterization, $g_\bullet$ is given for $\nu$ in edge I by $\alpha(g_\bullet\nu) = 50\alpha(\nu)/(5 + 36\alpha(\nu))$.
Thus by computations similar to those in Examples \ref{ex:mono4} and \ref{ex:mono5}, one sees $c(g^{n+2}, \nu) = 110c(g^{n+1},\nu) - 1000c(g^n,\nu)$ for $\nu$ in edge I.
In fact, since $c(g, \nu)\equiv 100$ for $\nu$ not in edge I, one also has vacuously that the relation $c(g^{n+2},\nu) = 110c(g^{n+1},\nu) - 1000c(g^n,\nu)$ holds for such $\nu$.
This proves that the sequence $c_n = c(f^n,\ord_0)$ satisfies the recursion relation $c_{n+8} = 110c_{n+4} - 1000c_n$.
\end{ex}

\begin{figure}[t]
\capstart
\begin{pspicture}(-7,-2)(7,3)
\psline(-6,-1)(6,-1)
\psline(-3,-1)(-3,2)
\psline(3,-1)(3,2)
\psdot(-6,-1)
\rput(-6,-1.3){$\ord_0$}
\psdot(6,-1)
\rput(6,-1.3){$\nu_{y,5/2}$}
\psdot(-3,-1)
\rput(-3,-1.3){$\nu_{y,5/4}$}
\psdot(3,-1)
\rput(3,-1.3){$\nu_{y,2}$}
\psdot(3,2)
\rput(3,2.3){$\nu_{y-x^2,5/2}$}
\psdot(-3,2)
\rput(-3,2.3){$\nu_{y^4 - x^5,41/32}$}
\rput(-4.5,-1.5){I}
\rput(3.5,.5){II}
\rput(-3.5,.5){III}
\rput(4.5,-1.5){IV}
\rput(0,-1.5){V}
\end{pspicture}
\caption{\label{fig}An illustration of the invariant finite tree $T\subset\mc{V}$ for the germ $f(x,y) = (y^2, y^4 - x^5)$ of \hyperref[ex:CAR]{Example~\ref*{ex:CAR}}, with edge labelings I-V. Note that the edge lengths are \emph{not} drawn to scale in the skewness parameterization.}
\end{figure}
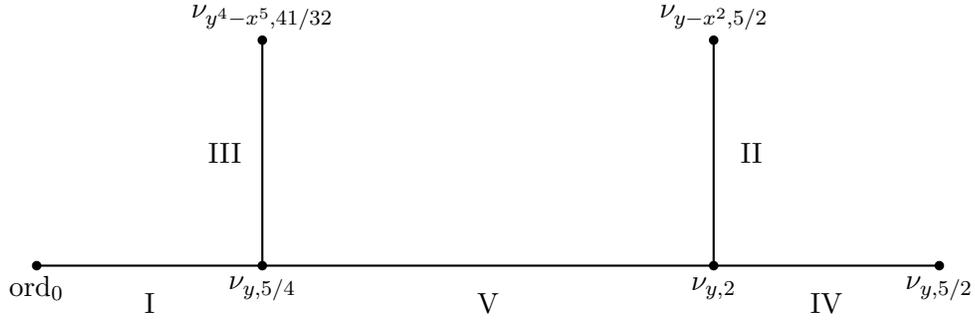

\begin{ex}\label{ex:high_order}
Let $\zeta\in \C$ be a primitive $m$-th root of unity, with $m\geq 2$.
We now consider the germ $f(x,y) = (\zeta x(x +  y^2), x + y^2)$.
There is exactly one eigenvaluation for $f$, namely the divisorial valuation $\nu_{x,2}$, which is in fact totally invariant for $f_\bullet$.
As we did in \hyperref[ex:segment]{Example~\ref*{ex:segment}}, one can identify the tangent space at $\nu_{x,2}$ with $\C\cup\{\infty\}$ by associating $\theta\in \C$ with the tangent vector in the direction of $\nu_{x - \theta y^2}$.
The tangent vector in the direction of $\ord_0$ corresponds to $\theta = \infty$.
Under this identification, the tangent map $df_\bullet$ at $\nu_{x,2}$ is given by the M\"{o}bius transformation $R(\theta) = \zeta\theta/(\theta + 1)$.
It is easy to check that $R$ has finite order $m$, and, in particular, that $\infty$ is a vector of period $m$.
In fact, if $\theta_n := R^n(\infty)$ for $n = 1,\ldots, m-1$, then the  edges $[\ord_0,\nu_{x,2}]$, $[\nu_{x,2},\nu_{x - \theta_
1y^2}]$, $\ldots$, $[\nu_{x,2},\nu_{x - \theta_{m-1}y^2}]$ are mapped into each other cyclically by $f_\bullet$.
One can show that
\[
\alpha(f_\bullet \nu) = \begin{cases} 1 + 2/\alpha(\nu) & \nu\in [\ord_0,\nu_{x,2}]\mbox{ or }[\nu_{x,2}, \nu_{x - \theta_{m-1}y^2}],\\ 
1 + \alpha(\nu)/2 & \nu\in [\nu_{x,2}, \nu_{x - \theta_j y^2}]\mbox{ for some }j = 1,\ldots, m-2,
\end{cases}
\]
and that
\[
c(f,\nu) = \begin{cases} \alpha(\nu) & \nu\in [\ord_0,\nu_{x,2}]\mbox{ or }[\nu_{x,2}, \nu_{x - \theta_{m-1}y^2}],\\ 
2 & \nu\in [\nu_{x,2}, \nu_{x - \theta_j y^2}]\mbox{ for some }j = 1,\ldots, m-2.
\end{cases}
\]
Using this data, it follows by the same arguments used in Examples \ref{ex:mono4}, \ref{ex:mono5}, and \ref{ex:CAR} that the sequence $c_n = c(f^n,\ord_0)$ satisfies the recurrence relation $c_{n} = (2^m+1)c_{n-m} + 2^mc_{n-2m}$ of order $2m$. 

It turns out, however, that $2m$ is not the minimal order of a linear recurrence satisfied by $c_n$.
It is possible to compute the $c_n$ using the above methods to be
\[
c_{km+r} = \begin{cases} \displaystyle2^{r-1}\frac{2^{(k+1)m}-1}{2^m - 1} & r = 1,\ldots, m-1, \bigskip\\
\displaystyle \frac{1}{2} + \frac{1}{2}\cdot \frac{2^{(k+1)m} - 1}{2^m - 1} & r = 0.\end{cases}
\]
The formal power series $\sum_{n=1}^\infty c_nt^n$ is exactly the Taylor expansion of
\[
\varphi(t) = \frac{1}{(1 - 2t)(1 + t + \cdots + t^{m-1})} - 1.
\]
It follows that the smallest order of a recursion relation satisfied by the sequence $c_n$ is $m$, and this recursion relation has characteristic polynomial $(t - 2)(1 + t + \cdots + t^{m-1})$.
\end{ex}

\begin{ex}\label{ex:high_order2}
Fix $d \geq 2$, $m\geq 1$ and let $\zeta$ be a primitive $(d^{m}-1)$-th root of unity.
We consider the germ $f(x,y)=(x^d (y-\zeta x), y^d (y-\zeta x))$.
The only eigenvaluation for this germ is $\ord_0$.
It follows immediately that $c(f^n,\ord_0) = (d+1)^n$, and hence that $c_n$ satisfies a linear recurrence of order $1$.
However, as we shall see, this will not be true for the sequence $c(f^n,\nu)$ for some $\nu\neq \ord_0$.

Identify the tangent space at $\ord_0$ with $\C\cup\{\infty\}$ by associating $\theta\in \C$ with the tangent vector in the direction of $\nu_{y - \theta x}$.
The tangent vector in the direction of $\nu_x$ corresponds to $\theta = \infty$.
With this identification, the tangent map $df_\bullet$ at $\ord_0$ is given by $R(\theta) = \theta^d$.
In particular, the tangent direction $\theta = \zeta$ is $m$-periodic, putting us in essentially the same situation considered in \hyperref[ex:high_order]{Example~\ref*{ex:high_order}}.
Using the same techniques, one can show that if $\nu = \nu_{y-\zeta x, \rho}$ for $\rho>1$, then the sequence $c_n:= c(f^n, \nu)$ satisfies the linear recurrence $c_n = ((d+1)^m + 1)c_{n-m} - (d+1)^mc_{n-2m}$ of order $2m$.
Once again, $2m$ is not the minimal order recurrence satisfied by the $c_n$.
The formal power series $\sum_{n=1}^\infty c_nt^n$ is the Taylor expansion of
\[
\varphi(t) = \frac{(d+1)t}{1 - (d+1)t} + \frac{(\rho - 1)t}{(1 - (d+1)t)(1 - t^m)},
\]
from which it follows that the minimal order of a linear recurrence satisfied $c_n$ is $m+1$, and its characteristic polynomial is $(t - (d+1))(t^m - 1)$.
\end{ex}

\begin{ex}\label{ex:not_stable}
In this example, we will consider the family of germs $f$ of the form
\[
f(x,y) = \big((y + \eps x^3)(cy + dx^2 + \eps x^3), y^2(ay+bx^2)(cy+dx^2)\big),
\]
where the constants $a,b,c,d,\eps\in \C$  satisfy the following properties:
\begin{enumerate}
\item[1.] $cd\eps\neq 0$. This ensures that $f$ is finite, and that the curves defined by $y$, $y + \eps x^3$, $cy + dx^2$, and $cy + dx^2 + \eps x^3$ are distinct.
\item[2.] The map $h(\theta) := (a\theta + b)/(c\theta + d)$ is a M\"{o}bius transformation such that $\infty$ is not preperiodic, and such that $h^N(\infty) = 0$ for some $N\geq 1$.
\end{enumerate} For such a germ $f$, the divisorial valuation $\nu_{y,2}$ is a globally attracting eigenvaluation; moreover, if we identify the tangent space at $\nu_{y,2}$ with $\pr^1$ by associating $\theta\in \C$ with the tangent direction defined by the curve $y - \theta x^2$, then the tangent map $df_\bullet$ at $\nu_{y,2}$ is given by the M\"{o}bius map $h(\theta)$.
One can compute that $c(f,-)$ is $\equiv 4$ in all the tangent directions $\theta \neq \infty, 0, -d/c$, and that $f^n_\bullet \ord_0$ lies in the tangent direction $h^n(\infty)$ for all $n\geq 0$.
The point $f^N_\bullet \ord_0$ lies in the segment $[\nu_{y,2}, \nu_{y,3}]$, and by using methods similar to those in the other examples, one can check that $\alpha(f^N_\bullet \ord_0) = 2 + 2^{-1}4^{1-N}$ and $c(f, f^N_\bullet\ord_0) = 2 + \alpha(f^N_\bullet\ord_0) = 4 + 2^{-1}4^{1-N}$.
Combining these data yields that
\[
c(f^n) = \begin{cases} 2\cdot 4^{n-1} & n \leq N,\\ (2\cdot 4^N + 1)4^{n-N-1} & n>N.\end{cases}\]
This sequence \emph{eventually} satisfies an integral linear recurrence formula (of order $1$), but does not satisfy any integral linear recurrence formula for all $n\geq 1$.
One way to check this is to note that if the sequence $c_n := c(f^n)$ satisfied a recurrence for all $n\geq 1$, then the formal power series $\sum_{n=1}^\infty c_nt^n$ would represent a rational function $P/Q$ with $\deg P \leq \deg Q$.
Direct computation, however, shows that that $\sum_{n=1}^\infty c_nt^n = (2t + t^{N+1})/(1-4t)$ does not satisfy this property.
In light of the observations made at the end of \S5, it follows that there does not exist a local algebraically stable model $\pi\colon X\to (\C^2,0)$ for $f$.
\end{ex}

\bibliographystyle{alpha}
\bibliography{biblio}

\end{document}